\documentclass[twoside]{article}

\usepackage{pgfplots}
\usepackage{amsfonts}
\usepackage{stmaryrd}
\usepackage{amssymb}
\usepackage{euscript}
\usepackage{amsthm}
\usepackage{amsmath}
\usepackage{amscd}
\usepackage{latexsym}
\usepackage{mathrsfs}
\usepackage{graphicx}
\usepackage{color}
\usepackage{dsfont}
\usepackage{bm}
\usepackage{multirow}
\usepackage{booktabs}
\usepackage{caption}
\usepackage{enumitem}%
\usepackage[algo2e,ruled,vlined]{algorithm2e} 
 \usepackage{algorithm}
 \usepackage{algorithmicx}
\usepackage{algpseudocode}
\usepackage{marvosym}
\usepackage{csquotes}
\usepackage{comment}

\usepackage[hidelinks]{hyperref}

\numberwithin{equation}{section}

\usepackage[top=2.5cm,lmargin=2.5cm,rmargin=2.5cm,bottom=2.5cm]{geometry} 

\newtheorem{theorem}{Theorem}[section]
\newtheorem{lem}{Lemma}[section]
\newtheorem{pro}{Proposition}[section]
\newtheorem{cor}{Corollary}[section]
\newtheorem{rem}{Remark}[section]
\newtheorem{rems}{Remarks}[section]
\newtheorem{ex}{Example}[section]
\newtheorem{defi}{Definition}[section]
\newtheorem{hyp}{Assumption}[section]

\newcommand{\bt}{\begin{theorem}}
\newcommand{\et}{\end{theorem}}
\newcommand{\bl}{\begin{lem}}
\newcommand{\el}{\end{lem}}
\newcommand{\bp}{\begin{pro}}
\newcommand{\ep}{\end{pro}}
\newcommand{\bcor}{\begin{cor}}
\newcommand{\ecor}{\end{cor}}

\newcommand{\bd}{\begin{defi} \rm }
\newcommand{\ed}{\end{defi}}
\newcommand{\brem }{\begin{rem} \rm }
\newcommand{\erem }{\end{rem}}
\newcommand{\brems }{\begin{rems} \rm }
\newcommand{\erems }{\end{rems}}
\newcommand{\bhyp }{\begin{hyp} \rm }
\newcommand{\ehyp }{\end{hyp}}
\newcommand{\bex}{\begin{ex} \rm }
\newcommand{\eex}{\end{ex}}

\newcommand{\wh}{\widehat}

\newcommand{\I}{\mathds{1}}

\newcommand{\llb}{\llbracket}

\newcommand{\esssup}{\operatornamewithlimits{ess\,sup}}
\newcommand{\essinf}{\operatornamewithlimits{ess\,inf}}

\newcommand{\cH}{\mathcal{H}}
\newcommand{\cS}{\mathcal{S}}

\newcommand{\cE}{{\mathcal E}}
\newcommand{\cF}{{\mathcal F}}

\newcommand{\cM}{{\mathcal M}}
\newcommand{\cT}{{\mathcal T}}

\newcommand{\FF}{{\mathbb F}}

\newcommand{\RR}{{\mathbb R}}

\newcommand{\PP}{{\mathbb P}}
\newcommand{\NN}{{\mathbb N}}
\newcommand{\EE}{{\mathbb E}}

\title{{\Large \bf  \uppercase{An Entropy Regularized BSDE Approach to Bermudan Options and Games}}\vskip35pt}

\author{Daniel Chee$\,^{a}$, Noufel Frikha$\,^{b}$ and Libo Li$\,^{a}$ \\ \\ \\ \\
\\ $^{a\,}$School of Mathematics and Statistics, University of New South Wales \\ Sydney, NSW 2052, Australia \\ \\
$^{b\,}$Universit\'e Paris 1 Panth\' eon-Sorbonne, Centre d’Economie de la Sorbonne, \\106 Boulevard de l’H\^opital, 75642 Paris Cedex 13, France\\ }

\begin{document}

\maketitle
\vskip20pt
\begin{abstract}

In this paper, we investigate optimal stopping problems in a continuous-time framework where only a discrete set of stopping dates is admissible, corresponding to the Bermudan option, within the so-called exploratory formulation. We introduce an associated control problem for the value function, represented as a non-c\`adl\`ag reflected backward stochastic differential equation (RBSDE) with an entropy regulariser that promotes exploration, and we establish existence and uniqueness results for this entropy-regularised RBSDE. We then compare the entropy-regularised RBSDE with the theoretical value of a Bermudan option and propose a reinforcement learning algorithm based on a policy improvement scheme, for which we prove both monotone improvement and convergence. This methodology is further extended to Bermudan game options, where we obtain analogous results. Finally, drawing on the preceding analysis, we present two numerical approximation schemes - a BSDE solver based on a temporal-difference scheme and neural networks and the policy improvement algorithm - to illustrate the feasibility and effectiveness of our approach.
\end{abstract}

\newpage
\tableofcontents
\newpage
\section{Introduction}

Optimal stopping problems play a central role in probability, statistics, and financial mathematics. Broadly speaking, such problems consist of choosing the most favourable time to stop a stochastic process in order to maximise a given expected reward. Classical examples include sequential testing in statistics (see, e.g., \cite{Wald,ShiryaevOS}), queueing and inventory models in operations research, and the pricing of American or Bermudan options in finance (see \cite{KS1998,P2009}). Formally, for some given positive reward process $P$ defined on a standard filtered probability space $(\Omega, \cF, \PP, \mathbb{F})$ where the filtration $\FF = (\cF_t)_{t\geq 0}$ is assumed to satisfy the usual conditions, the constrained optimal stopping problem and constrained optimal stopping games, see Kifer \cite{K2013, K2000}, is to compute
\begin{align}
V_t & =  \esssup_{\tau \in \mathcal{T}_{t,T}(S)} \mathbb{E}[P_\tau|\cF_t], \quad \mbox{ and } \quad  
V_t = \esssup_{\tau \in \mathcal{T}_{t,T}(S)} \essinf_{\sigma \in \mathcal{T}_{t,T}(S)} 
 \mathbb{E}[P_{\tau\wedge \sigma}|\cF_t], \label{ProblemIntro}
\end{align}

\noindent where, for a set $S\subseteq [0,T]$, $\cT_{t,T}(S)$ stands for the set of all $\FF$-stopping times taking values in $(S \cup \{T\}) \cap [t,T]$.  In the case where $S$ is a discrete set, we refer to the constrained optimal stopping problem in \eqref{ProblemIntro} as the {\it Bermudan (game) option}, and if $S = [0,T]$ we refer \eqref{ProblemIntro} as the {\it American (game) option}.

Traditional approaches to optimal stopping typically rely on dynamic programming and variational inequalities \cite{bensoussan1982applications,KS1998}, or on partial differential equations (PDEs) when the underlying dynamics are Markov processes \cite{fleming2006controlled,P2009}. While mathematically powerful, these methods require full knowledge of the underlying stochastic model and its reward structure, an assumption that is rarely satisfied in practice.

In parallel, reinforcement learning (RL) has emerged as a powerful framework for sequential decision-making under uncertainty \cite{bertsekas1996neuro,SB}. By learning from data through exploration and exploitation, RL can design near-optimal strategies without assuming explicit knowledge of the system dynamics. In continuous-time settings, entropy regularisation has proven particularly effective for balancing exploration and exploitation, leading to mathematically tractable formulations and provably stable algorithms \cite{haarnoja2017reinforcement,geist2019theory}.

Beyond these theoretical advances, a large body of research has focused on the numerical resolution of optimal stopping problems. Monte Carlo and machine learning approaches have been widely investigated, see for instance \cite{AKL2025, LL2021,BCJ2019,BCJW2021,BCJ2020,R2010,SZH2013,FMD2022,GLW2023,L2018,BG2004,LS2021,FD2021,R2002,RST2022,STG2023,DD2024,DSXZ2024,D2023,DFX2024}. Notable examples include the deep-learning algorithms of Becker, Cheridito, and Jentzen for pricing American-style options \cite{BCJ2019,BCJ2020}, the dual representation methods of Reppen, Soner, and Tan \cite{RST2022}, and the numerical schemes recently proposed by Soner and Tissot-Daguette \cite{STG2023}. More recently, motivated by RL and entropy regularisation, Dai and Dong \cite{DD2024}, Dai \emph{et al.}\ \cite{DSXZ2024}, Dong \cite{D2023}, and Dianetti \emph{et al.}\ \cite{DFX2024} have explored randomised stopping formulations.

The key idea underlying this line of work is the randomised stopping representation of optimal stopping problems (see Gyöngy and Šiška \cite{GS2008}). Rather than making a sharp \emph{stop-or-continue} decision, one introduces a stopping probability, thereby recasting the problem as a stochastic control problem in which the control variable is the instantaneous probability of stopping conditional on continuation. This perspective not only provides new analytical tools but also creates a natural bridge to modern RL algorithms, particularly in the face of sparse terminal rewards. Yet this representation comes with two important challenges. First, in the context of American options, the optimal stopping rule in discrete time (respectively continuous time) is inherently binary: the stopping probability takes values $0$ or $1$ (the stopping intensity takes values (the stopping intensity takes values $0$ or $\infty$). Second, this lack of smoothness is problematic both conceptually and numerically: (i) modern machine learning methods typically require differentiability of the control with respect to neural network parameters to exploit gradient descent, and (ii) discontinuous controls make numerical schemes more fragile and more sensitive to model misspecification.

To overcome these difficulties, recent research has explored alternative formulations and learning-based approaches. To address the first problem, Becker \emph{et al.} \cite{BCJ2019} proposed a two-step procedure. First, they introduce an intermediate neural network to parameterize the probability of stopping (the control), allowing gradient descent to be applied to the value function with respect to the neural network parameters. Then, to recover a candidate for the optimal control, the neural network output is passed through a function that maps it to either $0$ or $1$. In contrast, Reppen \emph{et al.} \cite{RST2022} and Soner and Tissot-Daguette \cite{STG2023} parameterise the exercise boundary itself using neural networks and smooth the stopping probability by expressing it as a smooth function of the distance between the value function and the learned boundary. This ensures that the control remains smooth in the neural network parameters.

More recently, Dong \cite{D2023}, Dai and Dong \cite{DD2024}, Dai \emph{et al} \cite{DSXZ2024}, and Dianetti \emph{et al.} \cite{DFX2024} have developed a continuous-time, diffusion-based framework using a PDE approach built upon the exploratory HJB equation of \cite{TZZ2021, WZZ2020}. By introducing entropy regularisation into the problem, they encourage exploration of the (partially unknown) environment, which in turn smooths the stopping intensity. A notable advantage of this method is that the regularisation strength, governed by the so-called temperature parameter $\lambda$, can be tuned to balance exploration and exploitation, thereby accelerating the learning process.

The present paper is, in spirit, closely related to the works of Dong \cite{D2023}, Dai and Dong \cite{DD2024}, Dai \emph{et al.} \cite{DSXZ2024}, and Dianetti \emph{et al.} \cite{DFX2024}, in that we likewise introduce entropy regularisation and develop reinforcement learning algorithms for optimal stopping problems and stopping games. Unlike these contributions, which adopt a PDE-based perspective and rely heavily on the structural properties of the underlying diffusion model, our approach follows a purely probabilistic route, closer in spirit to Becker \emph{et al.} \cite{BCJ2019}, and imposes little to no assumptions on the model dynamics. Furthermore, while the aforementioned studies focus mainly on American options, in either discrete or continuous time, our analysis centres on Bermudan options in continuous time. The extension of the present work to the American case is the subject of ongoing research \cite{CFL2025}.

Our approach builds upon the Doob-Meyer-Mertens decomposition of the value function associated with the optimal stopping problem \eqref{ProblemIntro}. This naturally leads us to a formulation in terms of backward stochastic differential equations (BSDEs), reflected BSDEs (RBSDEs), and double reflected BSDEs (DRBSDEs). Unlike \cite{D2023, DD2024, DSXZ2024, DFX2024}, who directly penalise the control (the stopping intensity in the American case or the stopping probability in the Bermudan case), we work at the level of the BSDE representation and introduce a penalised relaxed control problem. And in sharp contrast to Becker \emph{et al.} \cite{BCJ2019}, who parameterise the stopping rule with neural networks and enforce binary decisions via a projection step, our method does not rely on parameterising the control directly. Instead, we regularise the reflection process itself, which yields a formulation that is theoretically anchored in stochastic analysis and that avoids the architectural constraints of neural-network parameterisations.

From a conceptual standpoint, our method can therefore be seen as a way to regularise the reflection process appearing in the RBSDE representation of the optimal stopping problem, rather than the stopping rule. This distinction is crucial: the approach of Becker \emph{et al.} \cite{BCJ2019} is entirely model-free and numerical, but it depends on representing the stopping decision through a neural network, whose architecture must be fixed in advance. By contrast, our framework does not require such a parameterisation. It is grounded in stochastic analysis and BSDE theory, which makes it independent of neural network design choices and adaptable across different settings. A notable advantage of this approach is that it remains entirely probabilistic, so our results are not tied to specific model dynamics, and the numerical algorithm we develop is fully data-driven. Moreover, because our framework rests on the Doob–Meyer–Mertens decomposition and stochastic analysis tools, it can be applied with minimal changes to discrete-time models, continuous-time models, or mixed settings with finitely many admissible stopping times. Finally, the method could be naturally extended to settings with market frictions, or more generally to optimal stopping problems and games under nonlinear expectations, see for instance \cite{EKPPQ1997, GIOQ2020, QS2014}.\\

\noindent \textbf{Organization of the paper.}
\vskip5pt
The paper is organised as follows. Section 2 introduces the notations and stochastic analysis tools that will be used throughout. Section 3 is devoted to the Bermudan option: we introduce the entropy-regularised formulation, analyse the convergence of the associated approximate stopping rules as the regularisation parameter tends to zero, and establish error bounds of order $\lambda \ln(\lambda)$ between the regularised and original values. We also study the convergence of a policy improvement algorithm and connect our approach with the dual representation of Rogers to obtain upper bounds.

In Section 4, we extend this analysis to the Bermudan game option. Following a similar structure, we show that the entropy-regularised values and stopping rules converge to their classical counterparts, and we establish the convergence of the associated policy update algorithm.

Section 5 presents numerical experiments. We implement two algorithms: a BSDE solver based on a temporal-difference scheme, and the policy improvement algorithm introduced earlier. We illustrate that the entropy-regularised formulation provides accurate approximations for small values of $\lambda$ and compare our results with those of Becker \emph{et al.} \cite{BCJ2019} for Bermudan max-call options. Finally, the Appendix gathers auxiliary lemmas and technical estimates.

\section{Tools and Notations} \label{sec:2}
We work on the standard filtered probability space $(\Omega, \cF, \PP, \mathbb{F})$ where the filtration $\FF = (\cF_t)_{t\geq 0}$ is assumed to satisfy the usual conditions and thus all $\FF$-adapted martingales are assumed to be c\`adl\`ag. To place our work within the framework of stochastic calculus and backward stochastic differential equations, we shall make use of stochastic calculus for non-c\`adl\`ag semimartingales introduced in Lenglart \cite{LE} and Gal'c\v uk \cite{G1981}. For the purpose of exposition, we follow the language used in Gal'c\v uk \cite{G1981} and, unless otherwise stated, all stochastic processes in concern are optional semimartingales which are known to exhibit finite left and right limits (l\`agl\`ad). Given any optional semimartingale $X$ we denote by $X_-$ and $X_+$ the left and right limits of $X$, and by convention we set $X_{0-} = 0$. The left and right jumps of $X$ are denoted by $\Delta X= X - X_-$ and $\Delta^+ X = X_+-X$ respectively. Any process of finite variation $V$ can be decomposed into its right continuous part and left continuous part, respectively, by setting $V^g = \sum_{s< \cdot} \Delta^+ V_s$ and $V^r := V - V^g$. The right continuous part $V^r$ can be further decomposed into $V^r = V^c + V^d$, where $V^d = \sum_{s\leq \cdot} \Delta V_s$ and $V^c := V^r - V^d$. This gives us the decomposition $V = V^c + V^d+ V^g$. In addition, suppose $X$ is an optional semimartingale such that $\Delta X > -1$ and $\Delta^+X > -1$, then the optional stochastic exponential of $X$ denoted by $\mathcal{E}(X)$, is given by
\begin{gather*}
\mathcal{E}_t(X) = \exp\left\{ X_{t} - \frac{1}{2}\left< X^c,X^c\right>_{t} \right\}\prod_{0<s\leq t} (1+\Delta X_s)e^{-\Delta X_s}\prod_{0\leq s< t} (1+\Delta^+ X_s)e^{-\Delta^+ X_s}.
\end{gather*}
If $X$ is a c\`adl\`ag increasing process or process of finite variation such that $\Delta X < 1$ then
\begin{align*}
\mathcal{E}_{u,t}(X) = e^{X^c_t}\prod_{u<s\leq t} (1+\Delta X_s) \quad \mathrm{and} \quad \mathcal{E}_{u,t}(X_-) = e^{X^c_t}\prod_{u\leq s < t} (1+\Delta X_s).
\end{align*}

\noindent We now introduce the spaces of processes that are used throughout the paper. We denote by \\
\noindent $\bullet$ $\mathcal{O}(\FF)$ the set of all $\FF$ optional processes,\\
\noindent $\bullet$
 $\mathcal{M}$ the set of all $\FF$ martingales and $\mathcal{M}_{u.i.}$ the set of all $\FF$ uniformly integrable martingales, \\
 \noindent $\bullet$ $\mathcal{A}^+$ the set of all $\FF$ adapted increasing processes of integrable variation, \\
  \noindent $\bullet$ $\mathcal{D}$ the set of all semimartingale which is of class $(D)$. 

 In the rest of this work, we denote by $\mathcal{T}$ the set of $\mathbb{F}$ stopping times and, given that our goal is to study optimal stopping problems where the admissible set of stopping times is only allowed to take values in a finite discrete subset of $[0,T]$. Letting $S = \{t_0, t_1, \dots , t_N \}$, where $t_{N} < T$, we shall denote by $\mathcal{T}(S)$ the set of $\mathbb{F}$ stopping times taking value in $S\cup\{T\}$ and $\mathcal{T}_{t,T}(S)$ the set of $\mathbb{F}$ stopping times taking value in $(S\cup\{T\})\cap [t, T]$. As convention, we set $t_0 = 0$ and $t_{N+1} = T$. The number of remaining exercise dates from time $t>0$ is denoted by the left continuous process $N(t) := (N+1) - \max\{i: t_i < t\}$.  Furthermore, we denote by $\Pi$ the set of conditional probability densities supported on $[0,1]$, or more specifically, we set
\begin{align*}
\Pi = \left\{ \pi = (\pi_{t_i}(u))_{i=1,\dots, N} \, : (u,s) \mapsto \pi_{s}(u) \text{ is } \mathcal{B}([0,1])\otimes \cF_{s} \text{ measurable, } \pi_{t_i}(u) \geq 0 \text{ and } \int^1_0 \pi_{t_i}(u) du = 1\right\}
\end{align*}
and the mean of $\pi_{t_i}$ is denoted by $\mu_{{\pi}_{t_{i}}} := \int_{0}^{1} u {\pi}_{t_{i}}(u) \, du$.
Finally, we write $x\vee y = \max(x,y)$, $x^+ = \max(x,0)$ and $\binom{a}{b} = \frac{a!}{a!(a-b)!}$ for $a >b \in \mathbb{N}_+$.



\section{Bermudan Option}\label{Bermudan}

\subsection{Representation of the value of the Bermudan option}
To motivate our approach, we begin by presenting several equivalent formulations for the valuation of a Bermudan option. These representations prove particularly useful for the introduction of entropy regularisation techniques.
Let $S = \{t_0, t_1, \dots , t_N \}$ be a prescribed, admissible set of exercise dates, with $t_0 =0$ and $t_N < T$. Consider a Bermudan option associated with a non-negative, $\FF$-adapted reward process $P$ satisfying $\max_{u\in S\cup\{T\}} \mathbb{E}[P_{u}] <\infty$. The value process $V$ of the option is then given by the classical optimal stopping representation:
\begin{align}
V_t = \esssup_{\tau \in \mathcal{T}_{t,T}(S)} \mathbb{E}[P_\tau|\cF_t]. \label{Problem}
\end{align}

By the results of El Karoui \cite{EK1981} and Maingueneau \cite{M1978} on optimal stopping, together with Schweizer's analysis of Bermudan options \cite{S2002}, the Doob-Meyer-Mertens decomposition of $V$ takes the form
\begin{align}
V_t & = P_T - (M_T- M_t) + \sum_{t\leq  t_i < T}  (P_{t_i} - V_{t_i+})^+ \label{RBSDEbermuda}
\end{align}
where $M\in \mathcal{M}_{u.i.}$. This representation may be interpreted as a linear RBSDE with lower barrier given by the reward process $P$. 

As noted in the introduction, it is well known that classical optimal stopping problems can be reformulated as so-called randomised stopping problems; see, for instance, Theorem 2.1 Gy\"ongy and \v Si\v ska \cite{GS2008}. This perspective has been fruitfully exploited in the discrete-time setting by Becker \emph{et al.}\@ \cite{BCJ2019} and in the continuous-time case for American options by Dong \cite{D2023}. However, to the best of our knowledge, such representations have not yet been established in the literature for Bermudan options. For the reader’s convenience, and as a foundation for our subsequent developments, we therefore provide such a result here, using tools from stochastic calculus and the theory of BSDEs. The proof is deferred to Appendix \ref{proof:lem:randomized}.

\bl \label{randomized}
The value process $(V_t)_{0\leq t \leq T}$ of the optimal stopping problem  in \eqref{Problem} exhibits a randomized stopping representation given by
\begin{align}
V_t = \esssup_{\Gamma \in \Xi} \mathbb{E}[P_T\mathcal{E}_{t,T}(-\Gamma_-) + \int_{[t,T[} P_s \mathcal{E}_{t,s}(-\Gamma_-) d\Gamma_s|\cF_t],\label{RV}
\end{align}
\noindent where $\Xi$ is given by
\begin{align*}
\Xi = \left\{\Gamma \in \mathcal{A}^+  : \Gamma = \sum_{i=0}^N \Delta \Gamma_{t_i}\I_{\llb t_i,\infty \llb} \quad \mathrm{and}\quad  \Delta \Gamma_{t_i} \in [0,1)\right\}.
\end{align*}

\el

In the credit risk and actuarial science literature, the process $\Gamma$  is commonly referred to as the hazard process, with the associated \emph{Cox} time defined by 
$$
\sigma := \inf\{t\geq 0: 1- \mathcal{E}_t(-\Gamma_-) \geq U\}\wedge T,
$$ 

\noindent where $U \sim \mathrm{Uniform}[0,1]$ is a random variable independent of $\cF_\infty$, and $\mathcal{E}$ denotes the optional stochastic exponential satisfying 
\begin{align*}
\PP(\sigma \geq t |\cF_{t}) = \mathcal{E}_t(-\Gamma_-) = \prod_{0\leq t_i <t} (1-\Delta \Gamma_{t_i}) \qquad \mathrm{and} \qquad  \Delta\Gamma_{t_i} = \frac{\PP(\sigma  = t_i|\cF_{t_i})}{\PP(\sigma  \geq t_i|\cF_{t_i})}.
\end{align*}
Thus, Lemma \ref{randomized} allows us to reinterpret the optimal stopping problem as a randomised stopping problem, in which $\Gamma$ plays the role of a control variable. Unlike the classical binary control framework, where the decision is simply whether or not to stop, the control $\Gamma$ now encodes the conditional probability of stopping, given survival up to the current time.

The RBSDE given in \eqref{RBSDEbermuda} provides a probabilistic representation of the optimal stopping problem. Moreover, the solution pair $(V,M)$ to \eqref{RBSDEbermuda} also satisfies the following generalised BSDE:
\begin{gather}
V_t = P_T - (M_T- M_t) + \esssup_{\Gamma \in \Xi}  \sum_{t\leq  t_i < T}  (P_{t_i} - V_{t_i+})\Delta \Gamma_{t_i}. \label{randomizedBSDE}
\end{gather}
It is immediate that the optimal control $\Gamma^{\star}$ in \eqref{randomizedBSDE} is characterised by the relation 
$$
\Delta \Gamma^*_{t_i} = \I_{\{P_{t_i} > V_{t_i+}\}}, \quad i=1, \cdots, N,
$$ 

\noindent that is, stopping occurs precisely when the immediate payoff exceeds the continuation value. In this sense, the optimal decision rule is again binary: exercise the option whenever $P > V_+ $.

\brem
The work of Becker \emph{et al.} \cite{BCJ2019} builds on the randomized stopping formulation given in \eqref{RV}, and replaces the control process $\Gamma = (\Delta \Gamma_{t_i})_{1\leq i \leq  N}$ with a sequence of neural networks parameterized by $\theta$, denoted $\Gamma^\theta = (\Delta \Gamma^\theta_{t_i})_{1\leq i \leq N}$. A stochastic gradient descent algorithm is then employed to optimise over $\theta$ in order to approximate the solution to \eqref{RV}.
However, as previously discussed, the optimal control $\Gamma^\star$ is composed of indicator functions. Consequently, in \cite{BCJ2019}, the neural network output must be post-processed through a function that maps into $[0,1]$ to recover a valid stopping strategy. \erem

\subsection{Entropy Regularized Bermudan Option}\label{ERBoption}

Starting from the BSDE representation \eqref{randomizedBSDE} of the randomised stopping problem \eqref{RV}, we introduce the following entropy-regularised BSDE, defined by the dynamics

\begin{gather}
V^\lambda_t = P_T - (M^\lambda_T- M^\lambda_t) + \esssup_{\pi\in \Pi}  \sum_{t\leq  t_i < T}  \int^1_0 \Big\{(P_{t_i} - V^\lambda_{t_i+})u \pi_{t_i}(u) - \lambda \pi_{t_i}(u)  \ln(\pi_{t_i}(u) )\Big\} \, du, \label{BermudaV}
\end{gather}

\noindent where $\lambda \geq 0$ is a temperature parameter governing the degree of regularisation.

Notably, when $\lambda = 0$, the BSDE above reduces to the unregularised form \eqref{randomizedBSDE}, and the optimal control $\pi^{\star} \in \Pi$ degenerates into a Dirac measure at either $0$ or $1$. In this setting, the optimal control corresponds to a bang-bang control, where the decision to exercise is binary: either to stop immediately $\pi^\star(u)= \delta_1(u)$ when the payoff exceeds the continuation value, or not at all $\pi^\star(u)= \delta_0(u)$ otherwise. 

This binary nature limits exploration in the state space and leads to non-smooth behaviour in numerical methods. By contrast, the entropy-regularised formulation allows for probabilistic stopping strategies through soft controls, enabling a smoother decision structure that interpolates between stopping and continuation. This softening of the control not only facilitates numerical tractability but also enhances gradient-based learning methods when approximating optimal stopping policies.

We observe that the non-càdlàg BSDE \eqref{BermudaV} admits a unique solution, which can be constructed via backward induction. The key idea is to analyse the equation between successive exercise dates. Specifically, since there are no stopping opportunities in the interval $(t_N, T]$, the martingale representation theorem ensures the existence of a unique solution $(V^\lambda, M^\lambda)$ to \eqref{BermudaV} on this interval, given by
$$
V^\lambda_t= \mathbb{E}[P_T|\cF_t], \quad \mbox{ and }\quad M^\lambda_t := V^\lambda_t - V^\lambda_{t_N+}, \quad t \in (t_N, T],
$$

\noindent where we define $V^\lambda_{t_N+} = \mathbb{E}[P_T|\mathcal{F}_{t_N}]$. At the exercise date $t_N$, the value process satisfies
\begin{align}
V^\lambda_{t_N} = V^\lambda_{t_N+} + \esssup_{\pi\in \Pi}\int^1_0 \Big\{(P_{t_N} - V^\lambda_{t_N+})u \pi_{t_N}(u) - \lambda \pi_{t_N}(u)  \ln(\pi_{t_N}(u) )\Big\} \, du. \label{BermudaVjump}
\end{align}

\noindent The optimal density $\pi^\star_{t_N}$ achieving the supremum above is of Gibbs form and is given explicitly by
\begin{align}
\pi_{t_i}^*(u) = \frac{\frac{1}{\lambda}(P_{t_i}-V^\lambda_{t_i +})}{e^{\frac{1}{\lambda}(P_{t_i}-V^\lambda_{t_i +})}-1} e^{\frac{1}{\lambda}(P_{t_i}-V^\lambda_{t_i +})u}, \qquad u \in [0,1], \quad i=0, \cdots, N. \label{pistar}
\end{align}

Substituting the expression for $\pi^\star$ from \eqref{pistar} into \eqref{BermudaVjump}, we obtain
\begin{align*}
V^\lambda_{t_N} = V^\lambda_{t_N+} + (P_{t_N} - V^\lambda_{t_N+}) \Psi\left(\frac{P_{t_N} - V^\lambda_{t_N+}}{\lambda}\right) =  V^\lambda_{t_N+}  + \lambda \Phi\left(\frac{P_{t_N} - V^\lambda_{t_N+}}{\lambda}\right)
\end{align*}
where the functions $\Psi$ and $\Phi$ are given by
\begin{equation} \label{eq: Psi_Phi_def}
    \Psi\left(x\right) = \frac{1}{x}\ln\left(\frac{e^x-1}{x}\right) \quad \mathrm{and} \quad \Phi(x) = x\Psi(x), \quad x \in \mathbb{R}\backslash\{0\},
\end{equation}
with the conventions $\Psi(0) := \frac{1}{2}$ and $\Phi(0) := 0$. As shown via repeated application of l’Hôpital’s rule, both $\Psi$ and $\Phi$ admit continuous extensions to the entire real line.

Further technical properties of the functions $\Psi$ and $\Phi$, including continuity and growth estimates, are established in Appendix~\ref{sec:technical:results}.

Following the same reasoning, equation \eqref{BermudaV} can be solved recursively on each interval $(t_{i-1}, t_i]$, using the updated terminal condition $V^\lambda_{t_i}$ for each $i = 0, 1, \dots, N$. Full details of this construction are given in the proof of Proposition~\ref{prop: v_lambda_unique}.

\brem 
In the case of an American option, a similar approach can be adopted by setting $d\Gamma_t = \lambda_t \, dt$, where the process $\lambda$ is interpreted as a stopping intensity. In this continuous-time setting, the optional stochastic exponential reduces to the standard exponential, that is, $\mathcal{E}_t(-\Gamma-) = \exp\big(-\int_0^t \lambda_s ds\big)$. This allows us to formulate an analogous relaxed control problem, along with its entropy-regularised version, by replacing the summation over discrete exercise dates in \eqref{BermudaVjump} and \eqref{BermudaV}, respectively, with a time integral.
Heuristically, one may think of the densities $\pi$ as being supported on $\mathbb{R}_+$ rather than on the compact interval $[0,1]$. 
Further details and a rigorous treatment of the continuous-time formulation will be studied in our future work.
\erem

\bp \label{prop: v_lambda_unique}
Suppose that $\max_{u\in S\cup \{T\}}\EE[P_{u}] < \infty$. Then the BSDE 
\begin{gather} 
V^\lambda_t = P_T - (M^\lambda_T- M^\lambda_t) + \sum_{t\leq  t_i < T}  \lambda \Phi\left(\frac{P_{t_i} - V^\lambda_{t_i+}}{\lambda}\right),\label{eq: v_lambda_unique}
\end{gather}

\noindent admits a unique solution $(V^\lambda, M^\lambda)\in \mathcal{D} \times \cM_{u.i.}$.
\ep

\begin{proof}

We divide the proof into several steps.

\noindent \emph{Step 1: Preliminaries and notations}

To simplify notation, for each $j=0, \cdots, N$, define 
$$
S_t^j = t_{j+1} \wedge (t\vee t_{j})
$$ 

\noindent and set the terminal condition as $V^{\lambda,N+1}_{t_{N+1}} := P_{T}$. We construct a backward recursive sequence of processes indexed by $j$, starting from $j=N$, using the following definitions:
\begin{itemize}[leftmargin=15pt]
\item For $t \in (t_{j}, t_{j+1}]$,
\begin{align}
    V_{t}^{\lambda, j} & := \mathbb{E}[V_{t_{j+1}}^{\lambda,j+1}|\cF_{t}]. \label{eq: V^lambda j}
\end{align}
\item At the exercise time $t_j$,
\begin{align}
    V^{\lambda, j}_{t_j} & := V^{\lambda, j}_{t_j+} + (P_{t_j} - V^{\lambda, j}_{t_j+}) \Psi\left(\lambda^{-1}(P_{t_j} - V^{\lambda, j}_{t_j+})\right). \label{eq: V_j^lambda j}
\end{align}
\item The martingale increment is defined as
\begin{align}
    M_{t}^{\lambda, j} & := \EE[V^{\lambda,j+1}_{t_{j+1}}|\cF_{S^j_t}] - \EE[V^{\lambda,j+1}_{t_{j+1}}|\cF_{t_j}]. \label{eq: M^lambda j}
\end{align}
\end{itemize}

\noindent \emph{Step 2: Boundedness and Uniform Integrability.}

Using the assumption $\max_{u\in S\cup \{T\}}\EE[P_{u}] < \infty$ and the bound $|\Psi(x)| \leq 1$, it follows from \eqref{eq: V_j^lambda j} that 
$$
\EE[|V_{t_{j}}^{\lambda, j}|] \leq 2\EE[|V^{\lambda, j}_{t_{j}+}|] + \EE[|P_{t_{j}}|] \leq 2\EE[|V^{\lambda,j+1}_{t_{j+1}}|] + \EE[|P_{t_{j}}|],
$$
\noindent which by induction yields 
$$
\max_{j=0,\dots, N} \EE[V^{\lambda,j}_{t_j}] < \infty.
$$
We then define the global candidate solution $(V^\lambda, M^\lambda)$ as 
\begin{align*}
V^\lambda = V^{\lambda, 0}_0\I_{\{0\}} + \sum_{j=0}^N V^{\lambda, j}\I_{(t_j,t_{j+1}]} \quad \mathrm{and} \quad 
M^\lambda = \sum_{j=0}^N M^{\lambda, j}.
\end{align*}

In view of \eqref{eq: M^lambda j}, each $M^{\lambda, j}$ is the difference of conditional expectations, it is uniformly integrable, and hence so is the sum $M^{\lambda}$. We thus conclude that $M^{\lambda} \in \cM_{u.i.}$. We now show that $V^{\lambda}\in \mathcal{D}$. Let $\tau < \infty$ be any stopping time. Then, 
\begin{align*}
    |V_{\tau}^{\lambda}|
    &\leq V_{0}^{\lambda, 0} + \sum_{j=0}^{N} \EE[V_{t_{j+1}}^{\lambda, j+1} |\cF_{\tau}]\I_{(t_{j}, t_{j+1}]}(\tau) 
    \leq V_{0}^{\lambda, 0} + \sum_{j=0}^{N} \EE[V_{t_{j+1}}^{\lambda, j+1} |\cF_{S_{\tau}^{j}}].
\end{align*}
This shows that the family \{$V_{\tau}^{\lambda}, \tau \in \mathcal{T}\}$ is uniformly integrable, and hence $V^{\lambda} \in \mathcal{D}$. \\

\noindent \emph{Step 3: Verification of the BSDE.} 

It remains to verify that $(V^{\lambda}, M^{\lambda})$ is the unique solution to \eqref{eq: v_lambda_unique}. Fix $k \in \{0, \hdots, N\}$ and $t \in (t_{k}, t_{k+1}]$. We compute
\begin{align*}
    &P_T - (M^\lambda_T- M^\lambda_t) + \sum_{t\leq  t_i < T}  \lambda \Phi\left(\frac{P_{t_i} - V^\lambda_{t_i+}}{\lambda}\right) \\
    &= P_{T} - (M_{T}^{\lambda, N} + \sum_{j=0}^{N-1}M_{T}^{\lambda, j})  + (M_{t}^{\lambda, k} + \sum_{j=0}^{k-1}M_{t}^{\lambda, j}) + \sum_{t\leq  t_i < T}  \lambda \Phi\left(\frac{P_{t_i} - V^\lambda_{t_i+}}{\lambda}\right) \nonumber \\
    &=  V^{\lambda,N}_{t_{N}+} + (V_{t}^{\lambda, k} - V^{\lambda,k}_{t_{k}+})  - \sum_{j = k}^{N-1} \left\{V^{\lambda,j+1}_{t_{j+1}} - V_{t_{j}+}^{\lambda, j}\right\} + \sum_{j=k+1}^{N}  \lambda \Phi\left(\frac{P_{t_j} - V_{t_j+}^{\lambda, j}}{\lambda}\right) \nonumber \\
    &= \sum_{j=k+1}^{N} \left\{V_{t_{j}+}^{\lambda, j} + \lambda \Phi\left(\frac{P_{t_j} - V_{t_j+}^{\lambda, j}}{\lambda}\right) \right\} - \sum_{j=k+1}^{N} V_{t_{j}}^{\lambda, j} + V_{t}^{\lambda, k} = V_{t}^{\lambda}.\nonumber 
\end{align*}
A similar argument applies at $t = t_{0} = 0$. 
Suppose now that there exists another solution $(\tilde{V}^{\lambda}, \tilde{M}^{\lambda}) \in \mathcal{D} \times \mathcal{M}_{u.i.}$ to \eqref{eq: v_lambda_unique}. Then, by the uniqueness of conditional expectation and backward induction, the same recursion must hold. Thus, the pair $(\tilde{V}^{\lambda}, \tilde{M}^{\lambda})$ must coincide with $(V^{\lambda}, M^{\lambda})$, and uniqueness is established.
\end{proof}

\brem \label{rem: monotone_conv}
From Corollary \ref{cor: monotone_driver} and the comparison theorem, it follows that the solutions to \eqref{eq: v_lambda_unique} are monotonically decreasing in the temperature parameter $\lambda$. Consequently, as $\lambda \downarrow 0$ we have the monotone convergence 
$$
V^{\lambda} \uparrow V.
$$
This property will be made quantitative in the next result.
\erem

\begin{theorem}\label{bermudat1}
Let $\lambda < 1$ and $t \in [0,T]$. Then, it holds
\begin{align*}
0\leq V_t - V^\lambda_t \leq  N(t) (1 - \ln(1-e{^{-1}}) + \mathbb{E}[\ln(\overline{V}_{t,T}\vee 1)\,|\,\cF_t]) \left[\lambda - \lambda \ln (\lambda)\right],
\end{align*}
where 
$$
N(t) = (N+1)- \max\{i: t_i < t\}, \quad \overline{V}_{t,T} = \max_{t_i \in S \cap [t,T[} V_{t_i}.
$$
\end{theorem}

\begin{proof}

\noindent \emph{Step 1:} By the comparison theorem (or by direct backward induction), we have
$$
V \geq V^\lambda.
$$
From Lemma \ref{lemma1.1}, for $c>0$ and $0<\epsilon < 1$,
\begin{gather}
c\Phi(x/c) \leq x^+ \leq c\Phi(x/c) + c - c\ln(1-e^{-\epsilon/c})+ c[\ln(|x|)]^+ - c\ln(c). \label{eqlem1.1}
\end{gather}
Later, we will set  $c = \epsilon = \lambda$. 
\vskip3pt

\noindent \emph{Step 2: } We write
\begin{equation*}
  V_{t} - V_{t}^{\lambda} = \int_{]t, T]}d(M^{\lambda} - M)_{s} + \sum_{t \leq t_{i} < T}  (P_{t_i} - V_{t_i+})^+ - \lambda\Phi(\lambda^{-1}(P_{t_i} - V^\lambda_{t_i+})). 
\end{equation*}
%
Taking conditional expectations given $\cF_t$, and using the fact that $\Phi$ is decreasing and $V\geq V^\lambda$, yields
\begin{align}
 V_t - V^\lambda_t  & = \mathbb{E}\Big[\sum_{t_i \in [t,T[}  \Big[(P_{t_i} - V_{t_i+})^+ - \lambda\Phi(\lambda^{-1}(P_{t_i} - V^\lambda_{t_i+})) \Big] \, \Big|\, \cF_t\Big] \nonumber \\
 & \leq \mathbb{E}\Big[\sum_{t_i \in [t,T[}  \Big[(P_{t_i} - V_{t_i+})^+ - \lambda\Phi(\lambda^{-1}(P_{t_i} - V_{t_i+})) \Big] \, \Big|\, \cF_t\Big] \nonumber\\
 & \leq \mathbb{E}\Big[ \sum_{t_i \in [t,T[}  \Big[\lambda - \lambda\ln(1-e^{-1})+ \lambda[\ln(|P_{t_i} - V_{t_i+}|)]^+ - \lambda\ln(\lambda) \Big] \Big|\cF_t \Big] \nonumber\\
& \leq N(t) \Big(1 -\ln(1-e^{-1}) + \mathbb{E}\Big[\max_{t_i \in [t,T]} [\ln(V_{t_i})]^+\Big|\cF_t\Big] \Big)
					 \Big[ \lambda  - \lambda\ln(\lambda) \Big],
\end{align}
where we used \eqref{eqlem1.1} together with the fact that $V_{t_i} = \max(P_{t_i}, V_{t_i+})$ and $\lambda - \lambda \ln(\lambda) > \lambda$ for $\lambda < 1$. 
\end{proof}

The previous result quantifies the discrepancy between the entropy-regularised value process $V^\lambda$ and the true value process $V$ in terms of the temperature parameter $\lambda$. In particular, it allows for an immediate uniform bound when the reward process $P$ is bounded, as stated in the following corollary.
\bcor
If the reward process $P$ is bounded, then  
$$\sup_{t\leq T} \,(V_t- V^\lambda_t) \leq  (N+1)(1.5 + 
\ln(\|P\|_\infty\vee 1))[\lambda -\lambda \ln(\lambda)].
$$
\ecor
\subsection{Approximate Stopping Rules}\label{subs:stoppingrule}

In this section, given $V^\lambda$ for a fixed $\lambda \geq 0$, we study an approximate optimal stopping time for \eqref{Problem} and an approximate randomised stopping time for the corresponding randomised stopping problem in \eqref{RV} and their convergence, as $\lambda \downarrow 0$, to the optimal stopping time for the Bermudan option. To this end, we require the following assumption, which is usually satisfied in diffusion-based models.
\bhyp\label{ass:VP}
The set $\bigcup_{j=0}^{N} \{V_{t_j+} = P_{t_j}\}$ is of zero probability.
\ehyp

We first study the approximate optimal stopping time associated with the optimal stopping problem \eqref{Problem}. We recall that from classical results on Snell envelopes (see, e.g., El Karoui \cite{EK1981}), it is well known that for any fixed $t\in [0,T]$, the stopping time 
$$
\tau^*_t 
 = \inf\{t\leq s \in S \cup \{T\}: V_{s} = P_{s}\},
$$ 
\noindent is optimal for the original problem \eqref{Problem}. 
As convention, we set $t_{-1} := -\infty$ and $\prod_{\emptyset} = 1$. Then by using the fact that $\{V = P\} \cap S = \{V_+ \leq P \}\cap S$, we have for $t_{i-1} <  t \leq t_{i}$ and  $i = 0,1,\dots, N+1$,
\begin{align*}
\tau^*_t = \sum^{N+1}_{j=i} t_j \I_{\{V_{t_j+} \leq P_{t_j}\}} \prod^{j-1}_{k=i} \I_{\{V_{t_k+} > P_{t_k}\}}. 
\end{align*}
Next, given the entropy-regularised Bermudan option value $V^\lambda$, it is natural to consider, as an approximation of the optimal stopping time $\tau^*_t$, the stopping time 
$$\hat \tau^\lambda_t := \inf \{t \leq s \in S \cup \{T\} : V^\lambda_{s+} \leq P_{s}\}.
$$

Similarly, from the definition of $\hat \tau^\lambda$, we have for $t_{i-1} <  t \leq t_{i}$ and $i = 0,1,\dots, N+1$
\begin{align*}
\hat \tau^\lambda_t = \sum^{N+1}_{j=i} t_j \I_{\{V^\lambda_{t_j+} \leq P_{t_j}\}} \prod^{j-1}_{k=i} \I_{\{V^\lambda_{t_k+} > P_{t_k}\}}.
\end{align*}
\bp \label{prop: btau_conv}
Under Assumption \ref{ass:VP}, for any $t\geq0$, as $\lambda \downarrow 0$, the stopping time $\hat \tau^\lambda_t$ converges almost surely to the stopping time $\tau^*_t$.
\ep 
\begin{proof}
The result follows from the continuous mapping theorem (Mann-Wald theorem) and Theorem \ref{bermudat1}.
\end{proof}

We now study the approximate optimal randomised stopping time associated with the randomised stopping problem \eqref{RV}. For fixed $\lambda > 0$ and $t\geq 0$, we set
\begin{align} \label{def: randomized_time}
    \tau^\lambda_t:= \inf\{s\geq t : 1- \mathcal{E}_{t,s}(-\Gamma^{\lambda}_-) \geq U\} \wedge T,
\end{align}

\noindent where the hazard process is given by 
    $$
    \Gamma^{\lambda} = \sum_{i=0}^N \Delta \Gamma_{t_i}^{\lambda}\I_{\llb t_i,\infty \llb}, \quad \Delta \Gamma^\lambda_{t_i} = \Psi\left(\lambda^{-1}(P_{t_i} - V^\lambda_{t_i+})\right),
    $$
    the random variable $U \sim \mathrm{Unif}[0,1]$ is independent of $\cF_\infty$ and $\mathcal{E}$ denotes the optional stochastic exponential. To be precise, the randomised stopping time is the non-decreasing process $1- \mathcal{E}(-\Gamma^{\lambda})$.

Intuitively, $ \Delta \Gamma^\lambda_{t_i}$ represents the conditional probability of stopping at time $t_i$, given survival up to $t_i$. In the limit $\lambda \downarrow 0$,  the control becomes a bang–bang rule, taking values in $\left\{0, 1 \right\}$, corresponding to immediate exercise when $P_{t_i} > V_{t_i+}$ and continuation otherwise.


This suggests that the randomised stopping rule $\tau^\lambda_t$ should converge to $\tau^*_t$ almost surely in this limit. To establish this convergence result, we begin by analysing the limiting behaviour of the jump sizes $\Delta \Gamma^\lambda_{t_i}$ of the hazard process as $\lambda \downarrow 0$. This is formalised in the following auxiliary lemma, which plays a key role in proving the main theorem below.


\bl \label{lem: limpsi}
    Under Assumption \ref{ass:VP}, for every $t_{i} \in S$, 
    \begin{gather}
        \lim_{\lambda \rightarrow 0} \Delta \Gamma^\lambda_{t_i} = \I_{\{V_{t_i+} \leq P_{t_i}\}}, \quad a.s. \label{eq: limpsi}
\end{gather}
\el
\begin{proof}
By definition, $\Psi(x) = \frac{1}{x}\Phi(x)$ for all $x \in \RR$. Therefore, Lemma \ref{lem: psi_cdf} guarantees that
\begin{equation} \label{eq: limpsi_eq1}
    \lim_{\lambda \rightarrow 0} \Psi\left(\frac{P_{t_i} - V_{t_i+}}{\lambda}\right) = \lim_{\lambda \rightarrow 0} \frac{\lambda}{P_{t_i} - V_{t_i+}}\Phi\left(\frac{P_{t_i} - V_{t_i+}}{\lambda}\right) = \frac{1}{2}\I_{\{V_{t_i+} = P_{t_i}\}} + \I_{\{V_{t_i+} < P_{t_i}\}} = \I_{\{V_{t_i+} \leq P_{t_i}\}}, \quad a.s.
\end{equation}

\noindent where we used the L'H\^opital's rule and the fact that $\Phi'(0)=1/2$.

By letting $\lambda \rightarrow 0$, it follows from Lemma \ref{lipphi} and Theorem \ref{bermudat1} that
\begin{align*}
    &\frac{\lambda}{P_{t_i} - V^\lambda_{t_i+}}\Phi\left(\frac{P_{t_i} - V^\lambda_{t_i+}}{\lambda}\right) - \frac{\lambda}{P_{t_i} - V_{t_i+}}\Phi\left(\frac{P_{t_i} - V_{t_i+}}{\lambda}\right)\\
    &= \frac{\lambda}{P_{t_i} - V^\lambda_{t_i+}}\left[\Phi\left(\frac{P_{t_i} - V^\lambda_{t_i+}}{\lambda}\right) - \Phi\left(\frac{P_{t_i} - V_{t_i+}}{\lambda}\right)  \right] + \left[\frac{\lambda}{P_{t_i} - V^\lambda_{t_i+}}- \frac{\lambda}{P_{t_i} - V_{t_i+}}\right]\Phi\left(\frac{P_{t_i} - V_{t_i+}}{\lambda}\right),
\end{align*}
converges $a.s.$ to zero which in turn yields \eqref{eq: limpsi}. 
\end{proof}

\bt \label{thm: btau_conv}
Under Assumption \ref{ass:VP}, for any $t\geq0$, as $\lambda \downarrow 0$, the random time $\tau^\lambda_t$ converges almost surely to the stopping time $\tau^*_t$.
\et 

\begin{proof}
We present the argument for $t=0$; the general case $t>0$ follows identically.

\noindent \emph{Step 1:} Define, for $t_i \in S$, 
\begin{equation*}
    A^\lambda_{t_i} := \sum^{i}_{j=0}\mathcal{E}_{t_j}(-\Gamma^\lambda_-)\Delta \Gamma^\lambda_{t_j} = \sum_{j=0}^{i}\Delta \Gamma^\lambda_{t_j}\prod^{j-1}_{k=0} (1-\Delta \Gamma^\lambda_{t_k}).
\end{equation*}

\noindent Recalling \eqref{def: randomized_time}, we can write
\begin{equation} \label{eq: btau_Hlambda}
\tau_{0}^{\lambda}  = \sum_{i=1}^{N} t_i \I_{(A^\lambda_{t_{i-1}},A^\lambda_{t_i}]}(U) + T\I_{(A^\lambda_{t_{N}},1]}(U)= \sum_{i=1}^{N+1} t_i \I_{(A^\lambda_{t_{i-1}},A^\lambda_{t_i}]}(U),
\end{equation}

\noindent where $A^\lambda_{t_{0}} = 0$ and $A^\lambda_{t_{N+1}} = 1$.

\noindent \emph{Step 2:} Observe that Lemma \ref{lem: limpsi} guarantees that $A^\lambda_{t_i} \rightarrow \I_{\{\tau_{0}^* \leq t_i\}}$ a.s. for any $t_{i} \in S$. Hence, as $\lambda \downarrow 0$, $U - A^\lambda_{t_{i}} \rightarrow U - \I_{\{\tau_{0}^* \leq t_i\}}$ a.s. By the continuous mapping theorem (Mann-Wald theorem), 
\begin{equation*}
    \I_{(-\infty,0]}(U - A^\lambda_{t_{i}}) \rightarrow \I_{(-\infty,0]}(U - \I_{\{\tau_{0}^* \leq t_i\}}) = \I_{\{ \tau_{0}^* \leq t_{i}\}} \quad \text{a.s.},
\end{equation*}

\noindent since the set of discontinuity points of the function $x \mapsto \I_{(-\infty, 0]}(x)$ is given by 
\begin{align*}
\{U - \I_{\{\tau_{0}^* \leq t_i\}} = 0\} = (\{U = 0\} \cap \{\tau_{0}^* >  t_i\}) \bigcup\, (\{U = 1\} \cap \{\tau_{0}^* \leq  t_i\}), 
\end{align*}
and has probability zero. Similarly, as $\lambda \downarrow 0$, $\I_{(0,\infty)}(U - A^\lambda_{t_{i-1}}) \rightarrow \I_{\{\tau_{0}^{*} > t_{i-1}\}}$ $a.s.$ Hence, for all $t_{i} \in S$, we conclude that 
\begin{align*}
    \I_{(A^\lambda_{t_{i-1}},A^\lambda_{t_i}]}(U) &= \I_{(0,\infty)}(U - A^\lambda_{t_{i-1}}) \times \I_{(-\infty,0]}(U - A^\lambda_{t_{i}}) \\
    &\longrightarrow  \I_{\{\tau_{0}^{*} > t_{i-1}\}} \times \I_{\{ \tau_{0}^* \leq t_{i}\}} = \I_{\{\tau^* = t_i\}} \quad {a.s.}
    \end{align*}
Repeating the above for $i = N+1$ gives
\begin{equation*}
    \I_{(A^\lambda_{t_{N}},1]}(U) = \I_{(0, \infty)}(U - A_{t_{N}^{\lambda}}) \rightarrow \I_{(0, \infty)}(U - \I_{\{\tau_{0}^{*} \leq t_{N}\}}) = \I_{\{\tau_{0}^{*} > t_{N}\}} = \I_{\{\tau_{0}^{*} = T\}} \quad {a.s.}
\end{equation*}

\noindent \emph{Step 3:} The arguments of the previous step together with the identity \eqref{eq: btau_Hlambda} implies that 
$$
\lim_{\lambda \downarrow 0} \tau_{0}^{\lambda} = \tau_{0}^{*}, \quad a.s.
$$
The argument for $t>0$ is identical.
\end{proof}

\brem
We point out that the approximate optimal stopping time $\wh \tau^{\lambda}$ for \eqref{Problem} is a stopping time, while ${\tau}^\lambda$ for \eqref{RV} is not a stopping time. 
\erem

\subsection{Policy Improvement Algorithm}\label{PIA}


In this section, we investigate the associated policy improvement algorithm.
Such algorithms are of particular importance in stochastic control, as they provide a constructive, iterative method for refining a candidate policy towards optimality, often yielding faster convergence and improved numerical stability compared with direct optimisation.

For a fixed $\lambda > 0$ and a given conditional probability density process $\pi = (\pi_{t_i})_{1\leq i \leq N} \in \Pi$, we define
\begin{align*}
G(t_i, x, \pi_{t_i}) & := \int^1_0 \Big\{(P_{t_i} - x)u \pi_{t_i}(u) - \lambda \pi_{t_i}(u)  \ln(\pi_{t_i}(u) )\Big\}du,\\
\pi^{*}_{t_i}(x,u) & : = \textrm{argmax}_{\pi} G(t_i, x, \pi_{t_i}(u)) = \frac{\frac{1}{\lambda}(P_{t_{i}}-x)}{e^{\frac{1}{\lambda}(P_{t_{i}}-x)}-1} e^{\frac{1}{\lambda}(P_{t_{i}}-x)u},\quad  u \in [0,1].
\end{align*}

Given an initial policy $\pi^0 \in \Pi$ and the initial value function $V^{\lambda,0}_t := \mathbb{E}[P_T |\mathcal{F}_t]$ for $t \in [0,T]$, , we proceed iteratively. At step $n$, assuming the current data $(\pi^n, V^{\lambda,n}) \in \Pi\times \mathcal{D}$ are known, the policy update step is
\begin{align}
\pi^{n+1}_{t_i}(u) & : = \pi^{*}_{t_i}(V^{\lambda, n}_{t_i+},u)  = \frac{\frac{1}{\lambda}(P_{t_{i}}-V^{\lambda,n}_{t_{i}+})}{e^{\frac{1}{\lambda}(P_{t_{i}}-V^{\lambda,n}_{t_i+})}-1} e^{\frac{1}{\lambda}(P_{t_{i}}-V^{\lambda,n}_{t_i+})u}, \quad u \in [0,1], \label{eq: pin1}
\end{align}

\noindent followed by the policy evaluation step
\begin{equation} \label{eq: b_Vlamb_n}
    V^{\lambda,n+1}_t  = P_T - (M^{\lambda,n+1}_T - M^{\lambda,n+1}_t) + \sum_{t\leq t_i< T} G(t_i, V^{\lambda,n+1}_{t_i+} , \pi^{n+1}_{t_i}), 
\end{equation}

\noindent where 
\begin{equation} \label{eq: Gpi_def}
    G(t_{i}, V_{t_{i}+}^{\lambda, n+1}, \pi_{t_{i}}^{n+1}) = \lambda \Phi\left(\frac{P_{t_{i}} - V_{t_{i}+}^{\lambda, n}}{\lambda}\right) + (V_{t_{i}+}^{\lambda, n} - V_{t_{i}+}^{\lambda, n+1})\times \mu_{\pi_{t_{i}}^{n+1}}.  
\end{equation}

It is immediate from \eqref{eq: pin1} that $\pi^{n+1} \in \Pi$.
Moreover, by the comparison theorem for BSDEs, $V^{\lambda,n+1} \geq V^{\lambda,n}$ for all $n \geq 1$, showing that each iteration produces a non-decreasing sequence of value functions.
The next result formalises the well-posedness of the policy evaluation step.

\bp \label{prop: b_Vlam_n_integ} 
Given any initial data $(\pi^0, V^{\lambda,0}) \in \Pi\times \mathcal{D}$, the GBSDE defined by \eqref{eq: b_Vlamb_n} admits a unique solution $(V^{\lambda, n+1}, M^{\lambda, n+1}) \in \mathcal{D} \times \mathcal{M}_{u.i.}$ for all $n\geq 0$.
\ep
\begin{proof}
The proof is split into two steps 
\vskip1pt
\noindent \emph{Step 1:} The proof proceeds along the same lines as that of Proposition \ref{prop: v_lambda_unique}, adapted to the iterative structure of the policy improvement algorithm.
We define, backwards in time, the following sequence of processes:
\begin{align}
    V_{t}^{\lambda, n+1, j} & := \mathbb{E}[V_{t_{j+1}}^{\lambda, n+1, j+1}|\cF_{t}], \quad \text{for } t \in (t_{j}, t_{j+1}], \label{eq: V^lambda n j} \\
    V^{\lambda, n+1, j}_{t_j} & := V^{\lambda, n+1, j}_{t_j+} + G(t_{j}, V_{t_{j}+}^{\lambda, n+1, j}, \pi_{t_{j}}^{n+1}), \label{eq: V_j^lambda n j}\\
    M_{t}^{\lambda, n+1, j} & := \EE[V^{\lambda, n+1, j+1}_{t_{j+1}}|\cF_{S^j_t}] - \EE[V^{\lambda, n+1, j+1}_{t_{j+1}}|\cF_{t_j}]. \label{eq: M^lambda n j}
\end{align}
for $n \in \mathbb{N}$ and $j \in {0, \dots, N}$, where $S_t^j := t_{j+1} \wedge (t \vee t_{j})$.
We set $V_{t_{j}}^{\lambda, 0, j} := V_{t_{j}}^{\lambda, 0}$ and $V_{t_{N+1}}^{\lambda, n+1, N+1} := P_{T}$.

\noindent \emph{Step 2:} We then propose the following candidate solution:
\begin{align*}
    V^{\lambda, n+1} = V^{\lambda, n+1, 0}_0\I_{\{0\}} + \sum_{j=0}^N V^{\lambda, n+1, j}\I_{(t_j,t_{j+1}]} \quad \mathrm{and} \quad 
M^{\lambda, n+1} = \sum_{j=0}^N M^{\lambda, n+1, j}.
\end{align*}

We prove that $(V^{\lambda, n+1}, M^{\lambda, n+1}) \in \mathcal{D} \times \cM_{u.i.}$ for all $n \in \mathbb{N}$ by induction.
Since $V^{\lambda, 0} \in \mathcal{D}$, $\max_{u\in S\cup \{T\}}\EE[P_{u}] < \infty$, $|\Psi(x)| \leq 1$ and $\mu_{\pi^{1}} \leq 1$, from \eqref{eq: Gpi_def} and \eqref{eq: V_j^lambda n j}, it follows from \eqref{eq: Gpi_def} and \eqref{eq: V_j^lambda n j} that 
$$
\EE[|V_{t_{j}}^{\lambda, 1, j}|] \leq 2\EE[|V_{t_{j}+}^{\lambda, 1, j}|] + \EE[|P_{t_{j}}|] + 2\EE[|V_{t_{j}+}^{\lambda, 0, j}|]
$$

\noindent and therefore $\max_{j=0,\dots, N} \EE[|V^{\lambda, 1, j}_{t_j}|] < \infty$. In the same manner as in Proposition \ref{prop: v_lambda_unique}, we conclude that $(V^{\lambda, 1}, M^{\lambda, 1}) \in \mathcal{D} \times \cM_{u.i.}$ is the unique solution to \eqref{eq: b_Vlamb_n} for $n=0$.

Assume that for some $k \in \mathbb{N}$, $(V^{\lambda, k+1}, M^{\lambda, k+1}) \in \mathcal{D} \times \cM_{u.i.}$ solves \eqref{eq: b_Vlamb_n}.
Substituting $V^{\lambda, k+1}$ in place of $V^{\lambda, 0, j}$ in \eqref{eq: V^lambda n j}–\eqref{eq: M^lambda n j}, we obtain $V_{t_{j}}^{\lambda, k+2, j} \in L^{1}$ for all $j \in \{0, \dots, N+1\}$.

By repeating the arguments used in Proposition \ref{prop: v_lambda_unique}, we conclude that $(V^{\lambda, k+2}, M^{\lambda, k+2}) \in \mathcal{D} \times \cM_{u.i.}$ is the unique solution to \eqref{eq: b_Vlamb_n} for $n=k$.
The result follows by induction.
\end{proof}

\begin{rem}
The above backward construction is not only a proof device for establishing well-posedness; it also mirrors the dynamic programming principle underlying the policy improvement algorithm.
Each iteration can be interpreted as a refinement of the continuation value through a backward induction scheme, with the policy update step ensuring monotonic improvement of the value function.
This structural correspondence between the proof and the algorithm provides further intuition for the stability and convergence properties of the method.
\end{rem}

We now present the principal convergence result of this section. It establishes that, after a finite number of iterations of the policy improvement algorithm defined in \eqref{eq: pin1}-\eqref{eq: b_Vlamb_n}, the sequence of value functions stabilises to the true solution $V^{\lambda}$, and the associated policies coincide with the optimal control $\pi^*$ given in \eqref{pistar}. This finite convergence property highlights the efficiency of the algorithm: in contrast to general iterative schemes, which may require infinitely many steps to achieve convergence, here optimality is attained exactly after a bounded number of updates determined solely by the number of remaining exercise opportunities. 

\bt  \label{thm: policy_convergence}
For fixed $\lambda \geq 0$ and $t \in [0,T)$, $V_{t}^{\lambda,n} = V_{t}^\lambda$ almost surely for $n \geq N(t_{+})-1$.
\begin{proof}
Fix $t \in [0,T)$ and consider the increment
\begin{align} \label{eq: policy_convergence_eq1}
	V_{t+}^{\lambda, n+1} - V_{t+}^{\lambda,n} 
					& = -(M_{T}^{\lambda,n+1} - M_{T}^{\lambda,n}) + (M_{t}^{\lambda, n+1} - M_{t}^{\lambda,n}) \nonumber\\
					& \quad + \sum_{t < t_{i} < T} \left \{ G(t_{i}, V_{t_{i}+}^{\lambda, n+1}, \pi_{t_{i}}^{n+1}) - G(t_{i}, V_{t_{i}+}^{\lambda, n}, \pi_{t_{i}}^{n})  \right\}.
\end{align}

 \noindent Since $V^{\lambda,n+1} \geq V^{\lambda,n}$ and $ \Phi$ is non-decreasing, we may write
\begin{align} \label{eq: policy_convergence_eq2}
& G(t_{i}, V_{t_{i}+}^{\lambda, n+1}, \pi_{t_{i}}^{n+1}) - G(t_{i}, V_{t_{i}+}^{\lambda, n}, \pi_{t_{i}}^{n}) \nonumber \\
& = \lambda \Phi\left(\frac{P_{t_{i}} - V_{t_{i}+}^{\lambda, n}}{\lambda}\right) - \lambda \Phi\left( \frac{P_{t_{i}} - V_{t_{i}+}^{\lambda, n-1}}{\lambda}\right)  + (V_{t_{i}+}^{\lambda, n} - V_{t_{i}+}^{\lambda, n+1}) \mu_{\pi_{t_{i}}^{n+1}} - (V_{t_{i}+}^{\lambda, n-1} - V_{t_{i}+}^{\lambda, n}) \mu_{\pi_{t_{i}}^{n}} \nonumber \\
& \leq (V_{t_{i}+}^{\lambda, n} - V_{t_{i}+}^{\lambda, n-1}).
\end{align}

Taking conditional expectations in \eqref{eq: policy_convergence_eq1} and invoking \eqref{eq: policy_convergence_eq2} yields
\begin{equation} \label{eq: v_lambda_policy_diff}
	V_{t+}^{\lambda, n+1} - V_{t+}^{\lambda,n}  \leq \sum_{t < t_i < T} \mathbb{E}[V_{t_{i}+}^{\lambda, n} - V_{t_{i}+}^{\lambda, n-1} | \cF_t].
\end{equation}

Iterating \eqref{eq: v_lambda_policy_diff} gives
\begin{align*}
	0 \leq V_{t+}^{\lambda, n+1} - V_{t+}^{\lambda,n}  & \leq \sum_{t < t_i < T} \sum_{t_i< t^1_i < T} \sum_{t^1_i< t^2_i < T} \dots \sum_{t^{n-2}_i< t^{n-1}_i < T} \mathbb{E}[V_{t^{n-1}_{i}+}^{\lambda, 1} - V_{t^{n-1}_{i}+}^{\lambda, 0} | \cF_t],
\end{align*}
where $t < t_i < t_i^1 < t^2_i < \dots < t^{n-1}_i < T$. If $n \geq N(t+)-1$, the innermost sum is empty (since there is no $t_i^{n-1} \in \mathcal{S}$ strictly before $t_N$) and hence $V_{t+}^{\lambda, n} = V_{t+}^{\lambda, n+1}$ almost surely. Moreover, for such $n$, the driver satisfies
\begin{equation*}
	\sum_{t < t_{i}<T}G(t_{i}, V_{t_{i}+}^{\lambda, n}, \pi_{t_{i}}^{n}) = \sum_{t < t_{i}<T} \lambda \Phi\left(\frac{P_{t_{i}} - V_{t_{i}+}^{\lambda, n-1}}{\lambda}\right) + (V_{t_{i}+}^{\lambda, n-1} - V_{t_{i}+}^{\lambda, n}) \mu_{\pi_{t_{i}}^{n}} = \sum_{t < t_{i}<T} \lambda \Phi\left(\frac{P_{t_{i}} - V_{t_{i}+}^{\lambda, n}}{\lambda}\right),
\end{equation*}
since $V_{t_{i}+}^{\lambda, n-1} = V_{t_{i}+}^{\lambda, n}$ whenever $n \geq N(t_i +)-1 \geq N(t_{i}+)$.By Propositions \ref{prop: v_lambda_unique} and \ref{prop: b_Vlam_n_integ}, together with \eqref{eq: b_Vlamb_n}, this identifies $V_{t+}^{\lambda, n}$ with $ V_{t+}^{\lambda}$. 

Finally, for $t=t_{i} \in S$, we have
$$
V_{t_{i}}^{\lambda, n} = V_{t_{i+}}^{\lambda, n} + G(t_{i}, V_{t_{i}+}^{\lambda, n}, \pi_{t_{i}}^{n}),
$$

\noindent and the above equality for $V_{t_i +}^{\lambda, n}$ implies 
$$
V_{t_{i}}^{\lambda, n} = V_{t_{i}}^{\lambda, n+1} = V_{t_{i}}^{\lambda}
$$ 
\noindent for all $n \geq N(t_{i}+)-1$. 
\end{proof}
\et

As $t \mapsto N(t)$ is decreasing, taking $t=t_{0}$ in Theorem \ref{thm: policy_convergence} shows that $V^{\lambda, n}$ converges to $V^{\lambda}$ on $[0,T]$ for $n \geq N(t_{0+})-1 = N$. Moreover, by uniqueness of the solutions in Theorem \ref{thm: policy_convergence} for $V_{t_{i}+}^{\lambda, n}$ for $t_{i} \in S$, it follows that $\pi^{n} \rightarrow \pi^{*}$. We formalise this in the next corollary.

\bcor \label{cor: policy_convergence}
    For fixed $\lambda \geq 0$, if $n \geq N$ then $V^{\lambda,n} = V^\lambda$ and $\pi^n = \pi^*$ almost surely.
\ecor

\subsection{Upper Estimate of the Bermudan Option}

In this section, we investigate an upper-bound for the Bermudan option value by means of the dual representation of Rogers \cite{R2002}. Specifically, we consider
\begin{align*}
    V_t 
        = \esssup_{\sigma \in \cT_{t,T}} \mathbb{E}[P_\sigma \mid \cF_t] 
        = \inf_{N \in \mathcal{H}^1} \mathbb{E}\left[ \sup_{t \leq \sigma \leq T} \big(P_\sigma - N_\sigma\big) \,\middle|\, \cF_t \right] + N_t,
\end{align*}
where the infimum is taken over all martingales $N \in \cH^{1}$. It is well known that the minimiser coincides with the martingale $M$ in Doob-Meyer-Marten's decomposition of $V$. Proposition \ref{prop: v_lambda_unique} suggests that the martingale $M^{\lambda}$ given in \eqref{eq: v_lambda_unique} may serve as a practical initial approximation to \(M\), thereby yielding the computable upper-bound
\begin{align*}
    U^\lambda_t := \EE[\sup_{t\leq \sigma \leq T} (P_\sigma - M^\lambda_\sigma) \, |\, \cF_t ] + M^\lambda_t.
\end{align*}

Our aim is now to analyse the convergence behaviour of $U^\lambda$ towards $V$ as $\lambda \downarrow 0$.


\bl \label{lem: upper_martingale}
Let $\lambda < 1$ and $t \in [0,T]$. Then,
\begin{align*}
    |M_{t}^{\lambda} - M_{t}| \leq C^{P,V}_{t} \left[\lambda - \lambda \ln (\lambda)\right],
\end{align*}
where 
\begin{align}
    C^{P, V}_{t} &= N(t) (1 - \ln(1-e{^{-1}}) + \mathbb{E}[\ln(\overline{V}_{t,T}\vee 1)\,|\,\cF_t])  + (N+1)(1 - \ln(1-e{^{-1}}) + \mathbb{E}[\ln(\overline{V}_{0,T}\vee 1)])  \nonumber\\
    &\hspace{1em} + (N+1) \left(1 - \ln(1-e^{-1}) + \max_{t_{i} \in [0,T]}[\ln(V_{t_{i}})]^{+}\right) \nonumber\\
    &\hspace{1em }+ \sum_{0 \leq t_{i} < t}  N(t_{i}+) (1 - \ln(1-e{^{-1}}) + \mathbb{E}[\ln(\overline{V}_{t_{i+1},T}\vee 1)\,|\,\cF_{t_{i}}]),  \label{CPV}
\end{align}
with
$$
    N(t) = (N+1) - \max\{i : t_i < t\}, 
    \qquad
    \overline{V}_{t,T} = \max_{t_i \in S \cap [t,T[} V_{t_i}.
$$

\el
\begin{proof}
Subtracting the martingale terms in \eqref{RBSDEbermuda} and \eqref{eq: v_lambda_unique} yields
\begin{equation} \label{eq: upper_M_difference}
    |M_{t} - M_{t}^{\lambda}| \leq |V_{t} - V_{t}^{\lambda}| +|V_{0} - V_{0}^{\lambda}| + \left|\sum_{0 \leq t_{i} < t}(P_{t_{i}} - V_{t_{i}+})^{+} - \lambda \Phi\left(\frac{P_{t_{i}} - V_{t_{i}+}^{\lambda}}{\lambda}\right)\right|.  
\end{equation}
The first two terms on the right-hand side are controlled by Theorem \ref{bermudat1}.  

For the last term, monotonicity of $\Phi$ together with the fact that $V^{\lambda} \leq V$ allows us to bound it via Lemma \ref{lemma1.1}
\begin{align*}
    &\left|\sum_{0 \leq t_{i} < t}(P_{t_{i}} - V_{t_{i}+})^{+} - \lambda \Phi\left(\frac{P_{t_{i}} - V_{t_{i}+}^{\lambda}}{\lambda}\right)\right| \\
    &\leq \sum_{0 \leq t_{i} < t} \max \left\{(P_{t_{i}} - V_{t_{i}+})^{+} - \lambda \Phi\left(\frac{P_{t_{i}} - V_{t_{i}+}}{\lambda}\right), \lambda \Phi\left(\frac{P_{t_{i}} - V_{t_{i}+}^{\lambda}}{\lambda}\right) - \lambda \Phi\left(\frac{P_{t_{i}} - V_{t_{i}+}}{\lambda}\right)\right\} \\
    &\leq \sum_{0 \leq t_{i} < t} \left[\lambda - \lambda\ln(1-e^{-1}) + \lambda[\ln(|P_{t_{i}} - V_{t_{i}+}|)]^{+} - \lambda \ln (\lambda) \right] + \sum_{0 \leq t_{i} < t} |V_{t_{i}+} - V_{t_{i}+}^{\lambda}| \\
    &\leq (N+1) \left(1 - \ln(1-e^{-1}) + \max_{t_{i} \in [0,T]}[\ln(V_{t_{i}})]^{+}\right) \left[\lambda - \lambda \ln (\lambda)\right] \\
    &\hspace{1em }+ \sum_{0 \leq t_{i} < t}  N(t_{i}+) (1 - \ln(1-e{^{-1}}) + \mathbb{E}[\ln(\overline{V}_{t_{i+1},T}\vee 1)\,|\,\cF_{t_{i}}]) \left[\lambda - \lambda \ln (\lambda)\right].
\end{align*}

\end{proof}

\bt
Let $\lambda < 1$ and $t \in [0,T]$. Then, 
\begin{equation*}
    0 \leq U^{\lambda}_t - V_t \leq \left( \EE[\sup_{t \leq \sigma \leq T}C^{P,V}_{\sigma}|\cF_{t}] + C^{P, V}_{t} \right)[\lambda - \lambda \ln (\lambda)],
\end{equation*}

\noindent where the process $C^{P,V}$ is defined in $\eqref{CPV}$.
\et

\begin{proof}    
From the Doob–Meyer decomposition $V = M - A$, we have
\begin{align*}
    \EE[\sup_{t\leq \sigma \leq T} (V_\sigma - M^\lambda_\sigma) \, |\, \cF_t ] &= \EE[\sup_{t\leq \sigma \leq T} (M_\sigma - M^\lambda_\sigma - A_\sigma) \, |\, \cF_t ] \\
    &\leq \EE[\sup_{t\leq \sigma \leq T} (M_\sigma - M^\lambda_\sigma) \, |\, \cF_t ]  +  \EE[\sup_{t\leq \sigma \leq T}  - A_\sigma \, |\, \cF_t ].
\end{align*}

\noindent Since $A$ is increasing, the second term in the above inequality is bounded by $-A_t$. Therefore,
\begin{align*}
    U^\lambda_t &= \EE[\sup_{t\leq \sigma \leq T} (P_\sigma - M^\lambda_\sigma) \, |\, \cF_t ] + M^\lambda_t \\
        &\leq \EE[\sup_{t\leq \sigma \leq T} (V_\sigma - M^\lambda_\sigma) \, |\, \cF_t ] + M^\lambda_t \leq \EE[\sup_{t\leq \sigma \leq T} (M_\sigma - M^\lambda_\sigma) \, |\, \cF_t ]  - A_t + M^\lambda_t.
\end{align*}
Applying Lemma \ref{lem: upper_martingale} completes the proof.
\end{proof}

\section{Bermudan Game Option}\label{BermudanGame}
We now extend our methodology to the valuation of a constrained Dynkin game, see Kifer \cite{K2013}, in which two agents have opposing interests and can each choose when to terminate the game at a finite set of exercise times. Such problems arise naturally in the pricing of financial derivatives with game features, such as convertible bonds or Israeli options, where the holder may exercise early while the issuer retains the right to cancel the contract. Our focus will be on the Bermudan setting, where stopping is allowed only at the discrete set $S = \{t_0,\dots, t_N\}$. More precisely, let $R$ and $P$ be $\FF$-adapted reward processes satisfying 
$$
    R > P \ \textnormal{on } S, \quad R_T = P_T, \quad 
    \max_{u \in S \cup \{T\}} \mathbb{E}[P_u] < \infty, \quad
    \max_{u \in S \cup \{T\}} \mathbb{E}[R_u] < \infty.
$$

The value process is then given by
\begin{gather*}
V_t = \esssup_{\tau \in \mathcal{T}_{t,T}(S)} \essinf_{\sigma \in \mathcal{T}_{t,T}(S)}\EE[P_\tau\I_{\{\tau \leq \sigma\}} + R_\sigma\I_{\{\tau > \sigma\}}\,|\, \cF_t]
\end{gather*}
recalling that $\mathcal{T}_{t,T}(S)$ stands for the set of $S$-valued stopping times taking values in $[t,T]$.

It follows from standard arguments in the literature (see, e.g., Kobylanski \emph{et al.}\@ \cite{KQR2013} or Cvitanic and Karatzas \cite{CK1996}) that $V$ satisfies the following double reflected BSDE (DRBSDE):
\begin{align}
V_t & = P_T  - \int_{]t,T]} dM_s + \sum_{t\leq t_i < T} (P_{t_i} - V_{t_{i+}})^+ - (V_{t_i+} - R_{t_{i}})^+ , \label{DRBSDE}
\end{align}
where $M$ is the martingale component of the Doob–Meyer–Mertens decomposition of $V$.  

At any \(t \in [0,T]\), the optimal strategies for the maximiser and the minimiser are respectively given by 
\begin{equation}\label{optimal:stopping:times:game:option}
    \tau_t = \inf \{ s \geq t : V_s = P_s \}
    \quad \mathrm{and} \quad 
    \sigma_t = \inf \{ s \geq t : V_s = R_s \}.
\end{equation}

In direct analogy with Subsection \ref{ERBoption}, our objective here is to introduce an entropy-regularised version of \eqref{DRBSDE}, analyse the resulting stopping rules, and develop a corresponding policy improvement algorithm.

\subsection{Entropy Regularised Bermudan Game Option and its Stopping Rule}\label{BermudanGame1}
We now extend the entropy regularisation approach developed for the Bermudan option to the game option setting.  
The main challenge is that the value process \eqref{DRBSDE} involves \emph{two} strategic optimisations for each player-leading to a DRBSDE.  

Introducing entropy penalties on both sides not only smooths the optimisation but also yields explicit forms for the randomised stopping rules of the maximiser and minimiser, enabling a unified analysis of convergence as the regularisation parameter vanishes.

In direct analogy with the Bermudan option case, we first rewrite \eqref{DRBSDE} under a control representation
\begin{align*}
V_t & = P_T  - \int_{]t,T]} dM_s + \esssup_{\Gamma^1 \in \Xi}\sum_{t\leq t_i < T} (P_{t_i} - V_{t_{i+}})\Delta \Gamma^1_{t_i} - \esssup_{\Gamma^2 \in \Xi} \sum_{t\leq t_i < T} (V_{t_i+} - R_{t_{i}})\Delta \Gamma^2_{t_i}.
\end{align*}

\noindent Entropy regularisation is introduced by replacing the pointwise suprema with expected rewards penalised by the relative entropy of the control distribution, yielding the BSDE
\begin{equation}\label{bsde:game:option}
\begin{aligned}
V^\lambda_t & = P_T  - \int_{]t,T]} dM^\lambda_s + \esssup_{\underline{\pi}\in \Pi}\sum_{t\leq t_i < T} G(t_i, V^\lambda_{t_i+} , \underline{\pi}_{t_i}) - \esssup_{\overline{\pi}\in \Pi} \sum_{t\leq t_i < T} H(t_i, V^\lambda_{t_i+}, \overline{\pi}_{t_i}),
\end{aligned}
\end{equation}

\noindent where the generator terms $G$ and $H$ are given by
\begin{align*}
G(t_i, x, \pi_{t_i}) & := \int^1_0 \Big\{(P_{t_i} - x)u \pi_{t_i}(u) - \lambda \pi_{t_i}(u)  \ln(\pi_{t_i}(u) )\Big\}du,\\
H(t_i, x, \pi_{t_i}) & := \int^1_0 \Big\{(x-R_{t_i})u \pi_{t_i}(u) - \lambda \pi_{t_i}(u)  \ln(\pi_{t_i}(u) )\Big\}du.
\end{align*}

The corresponding optimal controls in \eqref{bsde:game:option} can be computed in closed form
\begin{align*}
\underline{\pi}^{*}_{t_i}(x,u) & : = \textrm{argmax}_{\pi} G(t_i, x, \pi_{t_i}(u)) = \frac{\frac{1}{\lambda}(P_{t_{i}}-x)}{e^{\frac{1}{\lambda}(P_{t_{i}}-x)}-1} e^{\frac{1}{\lambda}(P_{t_{i}}-x)u} \qquad u\in [0,1],\\
\overline{\pi}^{*}_{t_i}(x,u) & : = \textrm{argmax}_{\pi} H(t_i, x, \pi_{t_i}(u)) = \frac{\frac{1}{\lambda}(x-R_{t_{i}})}{e^{\frac{1}{\lambda}(x-R_{t_{i}})}-1} e^{\frac{1}{\lambda}(x-R_{t_{i}})u} \qquad u\in [0,1].
\end{align*}

Substituting $\underline{\pi}^*$ and $\overline{\pi}^*$ back into \eqref{bsde:game:option}, we obtain the entropy-regularised DRBSDE
\begin{align} \label{eq: v_lambda_game}
V^\lambda_t & = P_T  - \int_{]t,T]} dM^\lambda_s + \sum_{t\leq t_i < T} \lambda \Phi\left(\frac{P_{t_i} - V^\lambda_{t_{i+}}}{\lambda}\right) - \lambda\Phi\left(\frac{V^\lambda_{t_i+} - R_{t_{i}}}{\lambda}\right).
\end{align}

The well-posedness of \eqref{eq: v_lambda_game} follows from similar arguments to those used in Proposition \ref{prop: v_lambda_unique}. The proof of the following result is thus omitted.
\bp \label{prop: v_game_unique}
The BSDE \eqref{eq: v_lambda_game} admits a unique solution $(V^\lambda, M^\lambda)$ in $\mathcal{D} \times \cM_{u,i}$.
\ep

We now state the central result of this section, which quantifies the bias introduced by entropy regularisation. Although it is possible to derive below in Theorem \ref{thm: game_converge} an upper and a lower bound with two different constants, for the sake of symmetry and to preserve a clear financial interpretation, we shall focus on an upper bound for the absolute value expressed in terms of the price of a Bermudan option with reward process $2R$. This formulation highlights the natural way in which entropy regularisation alters the valuation by smoothing the exercise boundary, while still allowing for interpretable financial comparisons.
For the proof, it is convenient to introduce, for a given process $X$, the following hazard processes:
\begin{align} \Gamma^P(X) & := \sum_{i = 1}^N \Psi\left(\frac{P_{t_i} - X_{t_i}}{\lambda}\right)\I_{\llb t_i, \infty\llb} \qquad \text{and} \qquad \Gamma^R(X) := \sum_{i = 1}^N \Psi\left(\frac{X_{t_i} - R_{t_i}}{\lambda}\right)\I_{\llb t_i, \infty\llb}. \label{hazardPR} 
\end{align}

\bt \label{thm: game_converge}
For any $0< \lambda < 1$, we have
\begin{gather*}
|V^\lambda_t - V_t| \leq N(t) \big(1- \ln(1-e^{-1}) + \EE\big[\ln(\overline{V}_{t,T}^{2R}\vee 1) \big|\cF_{t} \big]\big)\left[\lambda  - \lambda \ln(\lambda)\right],
\end{gather*}
where $N(t) = (N + 1) - \max\{i: t_i < t\}$ and $\overline{V}^{2R}_{t,T} = \max_{t_i\in S \cap [t,T[} V^{2R}_{t_i}$, $V^{2R}$ being the price of a Bermudan option with payoff process $2R$. 
\et 

\begin{proof}

The argument proceeds in several steps.\\

\noindent\emph{Step 1: Set up and comparison.}  

For convenience, define $B(s,x) := (P_s - x)^+ - (x- R_s)^+$ and $
B^\lambda(s,x) := \lambda \Phi\left(\lambda^{-1}(P_{s} - x)\right) - \lambda\Phi\left(\lambda^{-1}(x - R_s)\right)$. We also introduce the intermediate drivers $B^{u,\lambda}(s,x) := (P_s - x)^+ - \lambda\Phi\left(\lambda^{-1}(x - R_s)\right)$ and $
B^{l,\lambda}(s,x)  := \lambda \Phi\left(\lambda^{-1}(P_{s} - x)\right) - (x- R_s)^+$ and denote by $V^{u,\lambda}$ and $V^{l,\lambda}$ the BSDE solutions associated with these drivers.

From Lemma \ref{lemma1.1} (see also Figure \ref{fig: driver_B}) we know that $B^{l,\lambda}(s,x)  \leq B(s,x) \leq B^{u,\lambda}(s,x)$ and $B^{l,\lambda}(s,x)   \leq B^\lambda(s,x) \leq B^{u,\lambda}(s,x)$.
By the comparison principle for BSDEs, this yields
$V^{l,\lambda} \leq V \leq V^{u,\lambda}$ and $V^{l,\lambda}  \leq V^\lambda \leq V^{u,\lambda}$. Hence,
\begin{equation}\label{two:sided:bounds:VminusVlambda}
V^{l,\lambda} - V^\lambda \leq V - V^\lambda \leq V^{u,\lambda} - V^\lambda.
\end{equation}

\medskip
\noindent\emph{Step 2: Upper estimate.} 

To obtain an upper estimate, we follow the same strategy as in Theorem \ref{bermudat1}.
Using the monotonicity of $\Phi$ and Lemma \ref{lemma1.1} with $c = \epsilon = \lambda$, we obtain
\begin{align*}
    0 \leq V^{u,\lambda}_t - V^\lambda_t  & =  \mathbb{E}\Big[ \sum_{t \leq t_i < T} B^{u,\lambda}(t_i,V^{u,\lambda}_{t_i+}) - B^\lambda(t_i,V^\lambda_{t_i+})      \, \Big| \,\cF_t \Big]\\
    &\leq \mathbb{E}\Big[ \sum_{t \leq t_i < T} (P_{t_{i}} - V^{u,\lambda}_{t_i+})^+  - \lambda \Phi\left(\lambda^{-1}(P_{t_{i}} - V^{\lambda}_{t_i+})\right)   \, \Big| \,\cF_t \Big]\\
    & \leq  \mathbb{E}\Big[ \sum_{t \leq t_i < T} (P_{t_{i}} - V^{u,\lambda}_{t_i+})^+ - \lambda \Phi\left(\lambda^{-1}(P_{t_{i}} - V^{u,\lambda}_{t_i+})\right)  \, \Big| \,\cF_t \Big]\\
    & \leq \mathbb{E}\Big[ \sum_{t_i \in [t,T[}  \Big[\lambda - \lambda\ln(1-e^{-1})+ \lambda[\ln(|P_{t_i} - V^{u,\lambda}_{t_i+}|)]^+ - \lambda\ln(\lambda) \Big] \Big|\cF_t \Big]. 
\end{align*}

We next compare $V^{u,\lambda}$ to the price of a defaultable Bermudan option with hazard process $\Gamma^R := \Gamma^R(V^{u,\lambda}_+)$, defined in \eqref{hazardPR}. To see this, let $\sigma_t = \inf\{t_i \in S\cap[t,T]: V^{u,\lambda}_{t_i+} \leq P_{t_i}\}$, in the case $\sigma_t < T$, it holds 
\begin{align*}
V^{u,\lambda}_t &= V^{u,\lambda}_{\sigma_t} - \int_{]t,\sigma_t]} dM^{u,\lambda}_s + \sum_{t \leq t_{i} < \sigma_t}(P_{t_{i}} - V_{t_{i}+}^{u, \lambda})^{+} - (V^{u,\lambda}_{t_{i}+} - R_{t_{i}}) \Psi\left(\frac{V^{u,\lambda}_{t_{i}+} - R_{t_{i}}}{\lambda}\right) \\
 &= V^{u,\lambda}_{\sigma_t} - \int_{]t,\sigma_t]} dM^{u,\lambda}_s + \int_{[t,\sigma_t[}(R_s-V^{u,\lambda}_{s+})d\Gamma^R_s.
\end{align*}
We observe that $0\leq V^{u,\lambda}_{\sigma_{t}+} \leq P_{\sigma_{t}} < R_{\sigma_{t}}$ and, at the point $\sigma_t$, we have $V^{u,\lambda}_{\sigma_t} =  V^{u,\lambda}_{\sigma_t+} \vee P_{\sigma_t}  - \lambda \Phi(\lambda^{-1}(V_{\sigma_t+}^{u,\lambda} - R_{\sigma_t}))$ so that, using the fact that $|\Phi'|\leq 1$ and $\Phi(0) = 0$, we deduce that 
\begin{align*}
V^{u,\lambda}_{\sigma_t} & =  V^{u,\lambda}_{\sigma_t+} \vee P_{\sigma_t}  - \lambda \Phi(\lambda^{-1}(V_{\sigma_t+}^{u,\lambda} - R_{\sigma_t}))  \\
&\leq P_{\sigma_t}  + \lambda \left|\Phi(\lambda^{-1}(V_{\sigma_t+}^{u,\lambda} - R_{\sigma_t})) - \Phi(0)\right| \\
&\leq P_{\sigma_t} + |R_{\sigma_t} - V_{\sigma_t+}^{u,\lambda}| \\
&\leq P_{\sigma_t} + R_{\sigma_t}.
\end{align*}
Hence,  by the It\^o-Lengart-Gal'c\v uk formula 
\begin{align*}
\mathcal{E}_t(-\Gamma^R_-) V^{u,\lambda}_t 
& = \mathbb{E}[ V^{u,\lambda}_{\sigma_t}\mathcal{E}_{\sigma_t}(-\Gamma^R_{\sigma_t}) + \int_{[t,\sigma_t[} R_s \mathcal{E}_s(-\Gamma^R_-) d\Gamma^R_s \,|\,\cF_t]\\
& \leq \esssup_{\sigma \in \mathcal{T}_{t,T}(S)}\mathbb{E}[ (R+P)_{\sigma}\mathcal{E}_{\sigma}(-\Gamma^R_{\sigma}) + \int_{[t,\sigma[} (R+P)_s \mathcal{E}_s(-\Gamma^R_-) d\Gamma^R_s \,|\,\cF_t].
\end{align*}

The inequality above shows that $V^{u,\lambda}$ is bounded from above by the price of a defaultable Bermudan option, whose value is necessarily no greater than that of its non-defaultable counterpart. In particular, we have 
$$
V^{u,\lambda}_t \leq \esssup_{\sigma \in \mathcal{T}_{t,T}(S)}\mathbb{E}[ (R+P)_{\sigma} |\,\cF_t] =: V^{R+P}_t.
$$

Moreover, since on the set $S$, the process satisfies $V^{R+P} = (R+P)\vee V^{R+P}_{+}$, we obtain an upper bound that depends on $V^{P+R} \leq V^{2R}$, which in turn is dominated by $V^{2R}$. Hence,
\begin{equation}\label{upper:estimate:diff:values}
    V_{t}^{u, \lambda} - V_{t}^{\lambda} \leq N(t) \Big(1- \lambda \ln(1-e^{-1}) + \EE\big[\max_{t_{i} \in [t, T[} \ln(V_{t_{i}}^{R+P} \vee 1) \big|\cF_{t} \big]\Big)\left[\lambda  - \lambda \ln(\lambda)\right].
\end{equation}

\medskip
\noindent\emph{Step 3: Lower estimate.} 

Similarly, for the lower bound, we have
\begin{equation}\label{first:ineq:lower:estimate}
\begin{aligned}
    0 \geq V^{l,\lambda}_t - V^\lambda_t  & =  \mathbb{E}\Big[ \sum_{t \leq t_i < T} B^{l,\lambda}(t_i,V^{l,\lambda}_{t_i+}) - B^\lambda(t_i,V^\lambda_{t_i+})  \, \Big| \,\cF_t \Big]\\
    & \geq  - \mathbb{E}\Big[ \sum_{t \leq t_i < T} (V^{l,\lambda}_{t_i+}- R_{t_i})^+  - \lambda\Phi\left(\lambda^{-1}(V^\lambda_{t_i+} - R_{t_i})\right)   \, \Big| \,\cF_t \Big] \\
    & \geq  - \mathbb{E}\Big[ \sum_{t \leq t_i < T} (V^{l,\lambda}_{t_i+}- R_{t_i})^+  - \lambda\Phi\left(\lambda^{-1}(V^{l, \lambda}_{t_i+} - R_{t_i})\right)   \, \Big| \,\cF_t \Big] \\
    &\geq - \mathbb{E}\Big[ \sum_{t \leq t_i < T}  \Big[\lambda - \lambda\ln(1-e^{-1})+ \lambda[\ln(|V^{l,\lambda}_{t_i+} - R_{t_{i}}|)]^+ - \lambda\ln(\lambda) \Big] \Big|\cF_t \Big].
\end{aligned}
\end{equation}

We then let $\nu_t = \inf\{t_i \in S\cap[t,T]: V^{l,\lambda}_{t_i+} \geq R_{t_i}\}$ and we use the fact that $|\Phi'|\leq 1$ and $\Phi(0) = 0$ to get 
\begin{align*}
    V_{\nu_{t}}^{l, \lambda} &= V_{\nu_{t}+}^{l, \lambda}\wedge R_{\nu_{t}} + \lambda \Phi\left(\frac{P_{\nu_{t}} - V_{\nu_{t}+}^{l, \lambda}}{\lambda}\right) 
    \leq R_{\nu_{t}} + |P_{\nu_{t}} - R_{\nu_{t}}| \leq 2R_{\nu_{t}}.
\end{align*}
We introduce the hazard process $\Gamma^P: = \Gamma^P(V^{l,\lambda}_+)$, recalling \eqref{hazardPR}, and observe that
\begin{align*}
V^{l,\lambda}_t = V^{l,\lambda}_{\nu_t} - \int_{]t,\nu_t]} dM^{l,\lambda}_s + \int_{[t,\nu_{t}[}(P_s-V^{l,\lambda}_{s+})d\Gamma^P_s.
\end{align*}

The It\^o-Lenglart-Gal'c\v uk formula gives
\begin{align*}
\mathcal{E}_t(-\Gamma^P_-) V^{l,\lambda}_t 
& = \mathbb{E}[ V^{l,\lambda}_{\nu_t}\mathcal{E}_{\nu_t}(-\Gamma^P_{\nu_t}) + \int_{[t,\nu_t[} P_s \mathcal{E}_s(-\Gamma^P_-) d\Gamma^P_s \,|\,\cF_t] \\
& \leq \esssup_{\nu \in \mathcal{T}_{t,T}(S)}\mathbb{E}[ 2R_{\nu}\mathcal{E}_{\nu}(-\Gamma^P_{\nu}) + \int_{[t,\nu[} 2R_s \mathcal{E}_s(-\Gamma^P_-) d\Gamma^P_s \,|\,\cF_t].
\end{align*}

The inequality above shows that $V^{l,\lambda}$ is bounded above by the price of a defaultable Bermudan option whose value is less than its non-defaultable counterpart. In particular, we have $V^{l,\lambda}_t \leq \esssup_{\sigma \in \mathcal{T}_{t,T}(S)}\mathbb{E}[ 2R_{\sigma} |\,\cF_t] =: V^{2R}_t$ and $V^{2R} = (2R)\vee V^{2R}_{+}$ on the set $S$. Hence, coming back to \eqref{first:ineq:lower:estimate}, we conclude that
\begin{equation}\label{lower:estimate:diff:values}
    - N(t) \Big(1 -\ln(1-e^{-1}) + \EE\Big[\max_{t_{i} \in [t, T[} \ln(V_{t_{i}}^{2R}\vee 1) \Big|\cF_{t} \Big]\Big)\left[\lambda - \lambda \ln(\lambda)\right] \leq V_{t}^{l, \lambda} - V^{\lambda}.
\end{equation}

Coming back to \eqref{two:sided:bounds:VminusVlambda} and combining \eqref{upper:estimate:diff:values} with \eqref{lower:estimate:diff:values} completes the proof.
\end{proof}

\begin{figure}[thp]
\begin{center}
\begin{tikzpicture}[scale = 1]
\begin{axis}[
        legend style={
            at={(0.7, 0.98)}, 
            anchor=north west,
            legend columns=1,
            column sep=1ex,
            draw=black, 
            inner sep=1ex, 
            cells={anchor=west}
        },
        xmax=10,ymax=8, samples=65,
    ytick={0}, 
    yticklabels={0},
      xticklabel=\empty,]
    \addplot[blue,thick] {ln( (exp((0-x)) -1)/(0-x) ) - max(x-3,0)};
    \addlegendentry{$B^{l, \lambda}$}
    
    \addplot[black,thick] {max(-x,0) - max(x-3,0)};
    \addlegendentry{$B$}
    
    \addplot[red,thick] {ln( (exp((0-x)) -1)/(0-x) )-ln( (exp((x-3)) -1)/(x-3) )};
    \addlegendentry{$B^{\lambda}$}
    
    \addplot[green,thick] {max(-x,0)-ln( (exp((x-3)) -1)/(x-3) )};
    \addlegendentry{$B^{u, \lambda}$}
    
\addplot[color=black, dotted] coordinates
{(-5, 0) (8, 0)};
\node[] at (axis cs: 0, -2.2) {{\footnotesize $P$}};
\addplot[color=black, dotted] coordinates
{(0, -5) (0, 7)};
\addplot[color=black, dotted] coordinates
{(3, -5) (3, 7)};
\node[] at (axis cs: 3, -2.2) {{\footnotesize $R$}};
\end{axis}
\end{tikzpicture}
\caption{Representative sketch of $B$, $B^\lambda$, $B^{l,\lambda}$ and $B^{u,\lambda}$ with $\lambda=1$.}
\label{fig: driver_B}
\end{center}
\end{figure}
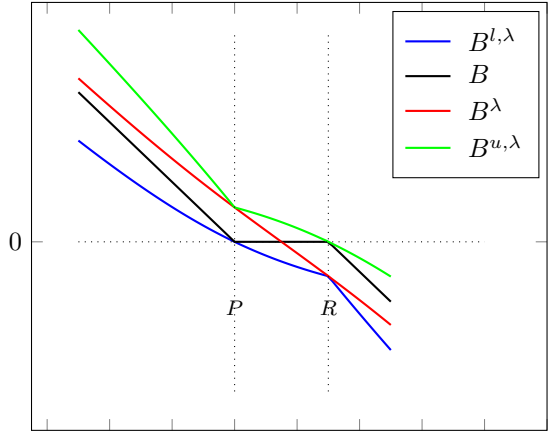

Next, we analyse the randomised stopping rules for the maximiser and minimiser of the game, and establish their convergence to the respective optimal stopping strategies. To this end, let $\lambda > 0$. Recalling from \eqref{eq: v_lambda_game} and \eqref{hazardPR}, we have
\begin{align}
V^\lambda_t & = P_T  - \int_{]t,T]} dM^\lambda_s + \int_{[t,T[} (P_{s} - V^\lambda_{s+})d\Gamma^P_s(V^\lambda_+) +\int_{[t,T[} (R_{s} - V^\lambda_{s+})d\Gamma^R_s(V^\lambda_+).
\end{align}
For convenience, set $\Gamma^{P,\lambda} := \Gamma^P(V^\lambda_+)$ and $\Gamma^{R,\lambda} := \Gamma^R(V^\lambda_+)$. We then define the random times. 
$$
\tau^\lambda_t := \inf\Big\{s \geq t : 1 - \mathcal{E}_{t,s}(-\Gamma^{P,\lambda}_-) \geq U_{1}\Big\} \wedge T,
\quad 
\sigma^\lambda_t := \inf\Big\{s \geq t : 1 - \mathcal{E}_{t,s}(-\Gamma^{R,\lambda}_-) \geq U_{2}\Big\} \wedge T,
$$
\noindent where $U_{1}$ and $U_{2}$ are independent $\mathrm{Uniform}([0,1])$ random variables, taken to be independent of $\cF_{\infty}$. We shall show that under the following assumption, as $\lambda \downarrow 0$, the pair $(\tau^\lambda_t, \sigma^\lambda_t)$ converges to the optimal stopping times $(\tau_t, \sigma_t)$, recalling \eqref{optimal:stopping:times:game:option}.

\bhyp\label{ass:VP2}
The sets $\bigcup_{j=0}^{N} \{V_{t_j+} = P_{t_j}\}$ and $\bigcup_{j=0}^{N} \{V_{t_j+} = R_{t_j}\}$ are of zero probability.
\ehyp

\begin{theorem}\label{stoppingrulegame}
For any $t\geq0$, as $\lambda \downarrow 0$, the pair $(\tau^\lambda_t,\sigma^\lambda_t)$ converges almost surely to $(\tau_t,\sigma_t)$.
\end{theorem}
\begin{proof}
 By Theorem \ref{thm: game_converge} and the proof of Lemma \ref{lem: limpsi}, as $\lambda \downarrow 0$, it holds
 $$
 \Delta\Gamma_{t_{i}}^{\lambda, R} \rightarrow \I_{\{V_{t_{i}+} > R_{t_{i}}\}}
 \qquad \mbox{and} \qquad
 \Delta\Gamma_{t_{i}}^{\lambda, P} \rightarrow \I_{\{V_{t_{i}+} \leq P_{t_{i}}\}}
 $$
 \noindent almost surely. Hence, applying the same line of reasoning as in Theorem \ref{thm: btau_conv}, we deduce that 
$
 (\tau_{t}^{\lambda}, \sigma_t^{\lambda}) \rightarrow (\tau_{t}, \sigma_t)$ almost surely as $\lambda \downarrow 0$.
\end{proof}


\subsection{Policy Iteration Algorithm}
In parallel with the Bermudan option case, we now consider the corresponding policy iteration algorithm for the game setting. Policy iteration algorithms play a central role in the analysis of stochastic games, as they provide a systematic procedure for refining the players’ strategies in an alternating fashion. For clarity of exposition, we fix $\lambda > 0$ and suppress the dependence on $\lambda$ in the notation. Recall that for $\pi \in \Pi$ and $t_i \in S$, we define
\begin{align*}
G(t_i, x, \pi_{t_i}) & := \int^1_0 \Big\{(P_{t_i} - x)u \pi_{t_i}(u) - \lambda \pi_{t_i}(u)  \ln(\pi_{t_i}(u) )\Big\}du,\\
\underline{\pi}^{*}_{t_i}(x,u) & : = \textrm{argmax}_{\pi} G(t_i, x, \pi_{t_i}(u)) = \frac{\frac{1}{\lambda}(P_{t_{i}}-x)}{e^{\frac{1}{\lambda}(P_{t_{i}}-x)}-1} e^{\frac{1}{\lambda}(P_{t_{i}}-x)u}
\end{align*}

\noindent and similarly
\begin{align*}
H(t_i, x, \pi_{t_i}) & := \int^1_0 \Big\{(x-R_{t_i})u \pi_{t_i}(u) - \lambda \pi_{t_i}(u)  \ln(\pi_{t_i}(u) )\Big\}du,\\
\overline{\pi}^{*}_{t_i}(x,u) & : = \textrm{argmax}_{\pi} H(t_i, x, \pi_{t_i}(u)) = \frac{\frac{1}{\lambda}(x-R_{t_{i}})}{e^{\frac{1}{\lambda}(x-R_{t_{i}})}-1} e^{\frac{1}{\lambda}(x-R_{t_{i}})u}.
\end{align*}

Therefore, starting from the initial condition $V^{0,0}_t = \mathbb{E}[P_T \,|\, \mathcal{F}_t]$, we define the policy update and evaluation steps as follows. Given a candidate policy $\overline{\pi}^{n}(u) \in \Pi$ for the minimiser, $\underline{\pi}^n \in \Pi$ for the maximiser, and an associated process $V^{n,n}\in \cS^1$, the policy update step for the minimiser is
\begin{align}
\overline{\pi}^{n+1}(u) = \overline{\pi}^{*}(V^{n,n}_+,u). \label{scheme1}
\end{align}

\noindent The corresponding policy evaluation step is then given by
\begin{align}
V^{n,n+1}_t = P_T - \int_{]t,T]} dM^{n,n+1}_s + \sum_{t\leq t_i < T}
G(t_i, V^{n,n}_{t_i+},\underline{\pi}^{n}_{t_i}) -
H(t_i,V^{n,n+1}_{t_i+},\overline{\pi}^{n+1}_{t_i}). \label{scheme2}
\end{align}

\noindent Next, the policy update step for the maximiser is defined by
\begin{align}
\underline{\pi}^{n+1}(u) = \underline{\pi}^{*}(V^{n,n+1}_+,u) \label{scheme3}
\end{align}

\noindent and the corresponding policy evaluation step is
\begin{align}
    V^{n+1,n+1}_t = P_T - \int_{]t,T]} dM^{n+1,n+1}_s + \sum_{t\leq t_i < T} 
    G(t_i,V^{n+1,n+1}_{t_i+},\underline{\pi}^{n+1}_{t_i})
- H(t_i,V^{n,n+1}_{t_i+},\overline{\pi}^{n+1}_{t_i}). \label{scheme4}
\end{align}

In other words, the iteration consists of a minimisation step in \eqref{scheme2} followed by a maximisation step in \eqref{scheme4}. By the comparison theorem for BSDEs, this procedure yields the inequalities
\begin{align}
V^{n+1,n+1} - V^{n,n+1} \geq 0   \qquad \mathrm{and} \qquad V^{n,n} - V^{n,n+1}\geq 0.\label{minmaxstep}
\end{align}

Before presenting the main result of this section, we first establish an auxiliary lemma providing upper estimates for the quantities appearing in \eqref{minmaxstep}.

\bl \label{prop: game_est}
    For $k \in \NN_{+}$ and $t \in [0,T)$, we have
    \begin{align*}
        V_{t}^{n+1, n+1} - V_{t}^{n, n+1} &\leq \sum_{t \leq t_{i} < t_{i}^{1} < \hdots < t_{i}^{2k-1} < T} \sum_{j=0}^{k}\binom{2k}{2j}\EE[V_{t_{i}^{2k-1}+}^{n+1-2k+j, n+1-2k+j} - V_{t_{i}^{2k-1}+}^{n-2k+j, n+1-2k+j}\, | \, \cF_{t}] \\
        &\hspace{1em} + \sum_{t \leq t_{i} < t_{i}^{1} < \hdots < t_{i}^{2k-1} < T} \sum_{j=0}^{k-1}\binom{2k}{2j+1}\EE[V_{t_{i}^{2k-1}+}^{n+1-2k+j, n+1-2k+j} - V_{t_{i}^{2k-1}+}^{n+1-2k+j, n+2-2k+j}\, | \, \cF_{t}].
    \end{align*}
\el

\bt \label{thm: policy_convergence_game}
For fixed $\lambda \geq 0$ and $t \in [0,T)$, the policy iteration scheme defined through \eqref{scheme1}\textendash\eqref{scheme4} converges almost surely, in the sense that 
$$
V_{t}^{n,n} = V_{t}^\lambda \quad \text{for all } n \geq N(t_{+}).
$$
\et
\begin{proof}
\noindent \emph{Step 1.} We begin with the decomposition
\begin{align*}
    V^{n+1,n+1} - V^{n,n} = \big(V^{n+1,n+1} - V^{n,n+1}\big) + \big(V^{n,n+1} - V^{n,n}\big).
\end{align*}
From \eqref{minmaxstep} we know that $V^{n+1,n+1} - V^{n,n+1} \geq 0$ and $V^{n,n+1} - V^{n,n} \leq 0$.  
Adopting the same approach as in Theorem \ref{thm: policy_convergence}, we exploit the monotonicity of the generator when comparing consecutive iterations. For clarity, we focus on the term $V^{n+1,n+1} - V^{n,n+1}$; the convergence of $V^{n,n+1} - V^{n,n}$ follows from an analogous argument.

\vskip3pt
\noindent \emph{Step 2.} Fix $t \in [0,T)$. For even $n$, Lemma \ref{prop: game_est} implies that
\begin{align*}
     V_{t+}^{n+1, n+1} - V_{t+}^{n, n+1} &\leq \sum_{t < t_{i} < t_{i}^{1}< \hdots < t_{i}^{n-1} < T} \sum_{j=0}^{n/2}\binom{n}{2j}\EE[V_{t_{i}^{n-1}+}^{1+j, 1+j} - V_{t_{i}^{n-1}+}^{j, 1+j}] \\
    &\hspace{1em} + \sum_{t < t_{i} < t_{i}^{1}< \hdots < t_{i}^{n-1} < T} \sum_{j=0}^{n/2-1}\binom{n}{2j+1}\EE[V_{t_{i}^{n-1}+}^{1+j, 1+j} - V_{t_{i}^{n-1}+}^{1+j, 2+j}].
\end{align*}
Note that for $n \geq N(t+)-1$ the terms on the right-hand side of the previous inequality satisfy 
\begin{equation*}
    \sum_{t_{i}^{n-2}< t_{i}^{n-1} < T} \sum_{j=0}^{n/2}\binom{n}{2j}\EE[V_{t_{i}^{n-1}+}^{1+j, 1+j} - V_{t_{i}^{n-1}+}^{j, 1+j}\, | \, \cF_{t}]=0,
\end{equation*}
and 
\begin{equation*}
    \sum_{t_{i}^{n-2} < t_{i}^{n-1} < T} \sum_{j=0}^{n/2-1}\binom{n}{2j+1}\EE[V_{t_{i}^{n-1}+}^{1+j, 1+j} - V_{t_{i}^{n-1}+}^{1+j, 2+j}\, | \, \cF_{t}]=0,
\end{equation*}
since no $t_N < t_{i}^{n-1} \in S$ exists, and the sums are therefore empty.  
In the case where $n \geq N(t+)-1$ is odd, the same conclusion follows from inequality \eqref{eq: game_est_eq1}. Consequently,
$$
V_{t+}^{n+1,n+1} = V_{t+}^{n,n+1} \quad \text{almost surely on } [0,T).
$$
 
By a parallel argument we also obtain $V_{t+}^{n,n+1} = V_{t+}^{n,n}$ almost surely on $[0,T)$ for all $n \geq N(t+)$.  
Hence, the driver of $V_{t+}^{n,n}$ satisfies
\begin{align*}
    &\sum_{t< t_i < T} G(t_i,V^{n,n}_{t_i+},\underline{\pi}^{n}_{t_i})
- H(t_i,V^{n-1,n}_{t_i+},\overline{\pi}^{n}_{t_i}) \\
&= \sum_{t< t_i < T} \lambda \Phi\left(\frac{P_{t_{i}} - V_{t_{i}+}^{n-1, n}}{\lambda}\right) - \lambda \Phi\left(\frac{V_{t_{i}+}^{n-1, n-1} - R_{t_{i}}}{\lambda}\right) + (V_{t_{i}+}^{n-1, n} - V_{t_{i}+}^{n,n})\mu_{\underline{\pi}_{t_{i}}^{n}} - (V_{t_{i}+}^{n-1,n} - V_{t_{i}+}^{n-1, n-1})\mu_{\overline{\pi}_{t_{i}}^{n}} \\
&= \sum_{t< t_i < T} \lambda \Phi\left(\frac{P_{t_{i}} - V_{t_{i}+}^{n, n}}{\lambda}\right) - \lambda \Phi\left(\frac{V_{t_{i}+}^{n, n} - R_{t_{i}}}{\lambda}\right)
\end{align*}

\noindent since for all $t_{i} > t$ and $n \geq N(t+)$ we have $V_{t_{i}+}^{n-1, n-1} = V_{t_{i}+}^{n-1, n} = V_{t_{i}+}^{n, n}$.  
By Proposition \ref{prop: v_game_unique}, it follows that $V_{t+}^{n,n}=V_{t+}^{\lambda}$ almost surely.

\noindent \emph{Step 3: } Finally, for each $t_{i} \in S$ we obtain
$$
V_{t_{i}}^{n, n} = V_{t_{i}+}^{n,n} + G(t_i,V^{n,n}_{t_i+},\underline{\pi}^{n}_{t_i}) - H(t_i,V^{n-1,n}_{t_i+},\overline{\pi}^{n}_{t_i}) = V_{t_{i}}^{\lambda} \quad \text{almost surely for all } n \geq N(t_{i}+).
$$

This completes the proof.
\end{proof}

As in Corollary \ref{cor: policy_convergence}, we now specialise Theorem \ref{thm: policy_convergence_game} to the case $t = t_{0}$.  
In this setting we observe that $V^{n,n}$ converges to $V^{\lambda}$ once $n \geq N(t_{0+}) = N+1$.  
Furthermore, by uniqueness of the solution we may conclude that the updated policies coincide with their optimal counterparts, i.e.~$\overline{\pi}^{n+1} = \overline{\pi}^{*}$ and $\underline{\pi}^{n+1} = \underline{\pi}^{*}$.  

\bcor \label{cor: policy_convergence_game}
   Let $\lambda \geq 0$ be fixed. Then, for all $n \geq N+1$, we have 
$$
V^{n,n} = V^{\lambda}, 
\qquad \overline{\pi}^{n+1} = \overline{\pi}^{*}, 
\qquad \underline{\pi}^{n+1} = \underline{\pi}^{*},
$$
almost surely.
\ecor



\section{Numerical implementation} \label{Numerics}

In this section, building on Proposition \ref{prop: v_lambda_unique}, we present two numerical algorithms for computing the price of a Bermudan option as formulated in \eqref{Problem}. Specifically, we approximate the value function $V$ by evaluating $V^\lambda$. From Remark \ref{rem: monotone_conv} and Theorem \ref{bermudat1}, it follows that the estimates of $V^\lambda$ improve monotonically as the temperature parameter $\lambda$ decreases towards zero. To illustrate the approximation of the unique solution to the BSDE \eqref{eq: v_lambda_unique}, we work within a Markovian framework. It should be emphasised, however, that the methodology developed here is not restricted to this setting. 

The first method is a backward algorithm that minimises the Temporal Difference (TD) error (see, for example, \cite{SB}), and may be viewed as a BSDE solver. We refer to this approach as a \emph{TD-based BSDE solver}. The second method is the policy improvement algorithm for $V^\lambda$, presented and analysed in Subsection \ref{PIA}.


\subsection{Description of the Numerical Algorithms}

\subsubsection{TD-based BSDE solver}
From \eqref{eq: v_lambda_unique}, we deduce that the process
$$
M^{\lambda}_t = V^{\lambda}_t + \int_{[0,t)} (P_s-V^{\lambda}_{s+}) d\Gamma^{\lambda}_s,
$$
\noindent is a c\`adl\`ag martingale, with the hazard rate process given by
    $\Gamma^{\lambda} = \sum_{i=0}^N \Delta \Gamma_{t_i}^{\lambda}\I_{\llb t_i,\infty \llb}$ where the jumps size is given by $\Delta \Gamma^\lambda_{t_i} = \Psi\left(\lambda^{-1}(P_{t_i} - V^\lambda_{t_i+})\right)$.
This implies that $\Delta^+M^{\lambda} = 0$, which in turn gives, for each $t_i \in S$,
\begin{equation} \label{eq: adjustment_condition}
    \Delta^+ V_{t_{i}}^{\lambda} + (P_{t_{i}}-V_{t_{i}+}^{\lambda}) \Psi\left(\frac{P_{t_{i}}-V_{t_{i}+}^{\lambda}}{\lambda}\right) = 0, 
\end{equation}
where $ \Delta^+ V_{t_{i}}^{\lambda} = V_{t_{i}+}^{\lambda} -  V_{t_{i}}^{\lambda}$, $V_{t_{i}+}^{\lambda} = \mathbb{E}[V^{\lambda}_{t_{i+1}} |\cF_{t_i}]$, and $V^{\lambda}_{T}  = P_T$. Moreover, \eqref{eq: V^lambda n j} and Theorem \ref{thm: policy_convergence} guarantee that for any $i\in \left\{0, \cdots, N-1\right\}$ and for any $t \in (t_i, t_{i+1}]$
\begin{gather}
V^{\lambda}_t = \mathbb{E}[V^{\lambda}_{t_{i+1}}|\cF_{t}]. \label{eq:martingale2}
\end{gather}



To proceed, fix $\lambda \in (0, 1)$ and let $v^{\lambda}$ be a function on $[0,T]\times \mathbb{R}^d$, with $V_t^{\lambda} = v^{\lambda}(t, X_t)$. We approximate $v^{\lambda}$ by $N+2$ components, namely the $N+1$ continuation values $V_{t_i+}^{\lambda}$ are estimated by neural networks $\mathscr{V}^{\eta_i}$ for each $t_i \in S = \{t_{0}, \dots, t_{N}\}$, and the value function between exercise dates is estimated by a family of neural networks $v^{\lambda, \theta}$, parameterized by $\theta \in \Theta \subset \mathbb{R}^H$. Moreover, from \eqref{eq: adjustment_condition}, we obtain an approximation for $v^{\lambda}$ at each exercise date by
\begin{equation} \label{eq: adjustment_condition_network}
    v^{\lambda, \eta_{i}}(t_{i}, X_{t_{i}}) = \mathscr{V}^{\eta_{i}}(X_{t_{i}}) + (P_{t_i}(X_{t_{i}}) - \mathscr{V}^{\eta_{i}}(X_{t_{i}})) \Psi\left(\frac{P_{t_i}(X_{t_{i}}) - \mathscr{V}^{\eta_{i}}(X_{t_{i}})}{\lambda}\right).
\end{equation}
Equation \eqref{eq: adjustment_condition_network} shows that the value function at an exercise date can be fully determined once the corresponding continuation value is known. 

The {\it adjustment condition} \eqref{eq: adjustment_condition} and the {\it martingale condition} \eqref{eq:martingale2} give our formulation of the TD (Temporal Difference) errors to compute the value function $V^\lambda$ and the continuation value $V^\lambda_+$ at each exercise date:

\begin{itemize}[leftmargin=15pt]


\item 
Guided by the relation $V_{t_i+}^{\lambda} = \mathbb{E}[V^{\lambda}_{t_{i+1}} \mid \cF_{t_i}]$, $i = 0, \dots, N$, we approximate the continuation value at time $t_i$ by solving
\begin{equation} \label{loss1}
    \min_{\eta_i }\mathbb{E}\Big[\big|v^{\lambda, \eta_{i+1}}(t_{i+1}, X_{t_{i+1}}) - \mathscr{V}^{\eta_i}(X_{t_i})\big|^2\Big], \quad i =N, \cdots, 0,
\end{equation}
\noindent with the terminal condition $v^{\lambda, \eta_{N+1}}(t_{N+1}, x) = P_T(x)$. Hence, starting from maturity, we proceed backwards in time, using \eqref{eq: adjustment_condition_network} to recursively estimate the continuation values at earlier exercise dates. Since the networks are trained backward in time, note that the optimal neural network $\mathscr{V}^{\eta^{\star}_i}$ obtained as the solution to \eqref{loss1} depends implicitly on all the optimal parameters $\eta^{\star}_{i+1}, \cdots, \eta^{\star}_{N}$ obtained in the previous optimization procedure but, to simplify notation, we omit these parameters in the notation.

    \item Guided by the {\it martingale condition} \eqref{eq:martingale2}, we approximate the value function between exercise dates by minimising the martingale loss
    \begin{equation}
        \mathbb{E}\left[ \sum_{i=0}^{N-1} \int_{t_i}^{t_{i+1}} (v^{\lambda, \eta^{\star}_{i+1}}(t_{i+1}, X_{t_{i+1}})  - v^{\lambda, \theta}(t, X_t))^2 \, dt  \right], \label{loss2}
    \end{equation}
    with respect to the parameter $\theta$. 
 \end{itemize}

 The pseudo-code for the above algorithm is described in Algorithm \ref{alg: offalgo_2}.

\begin{algorithm2e}[H] \label{alg: offalgo_2} 
\DontPrintSemicolon 
\SetAlgoLined 
\vspace{1mm}
{\bf Input data}: Number of episodes $M$, exercise dates $0 = t_0 <  \ldots < t_N < t_{N+1} = T$, number of mesh time-grid $n>>N$ (time step $\Delta t$ $=$ $T/n$, the dates $\hat{t}_k=k \Delta t$, $k=0, \cdots, n$ are compatible with the exercise dates $t_i$), learning rates $\rho_G$, $\rho_E$ for the two steps, $R$ sample paths, batch size $L<<R$. 
Parameter $\lambda$ for entropy regularization. 
Neural networks $v^{\lambda, \theta}$ and $(\mathscr{V}^{\eta_{i}})_{i = 0, 1, \dots, N}$ of the value function and conditional expectations. 

{\bf Initialization}: $v^{\lambda, \theta}({t}_{N+1}, x) = P_T(x)$,  
$v^{\lambda, \eta_{N+1}}({t}_{N+1}, x) = P_T(x)$. Generate $L$ paths $(X_{\hat{t}_{k}}^{(l)})_{0\leq k \leq  n}$, $l=1, \cdots, L$ of $X$.
\\
\For{exercise dates ${t}_{i} = {t}_{N},\ldots, {t}_{0}$}
{ 
    \For{each epoch $m = 1, \hdots, M$}
	   {
       Compute the gradient $\nabla_{\eta_{i}}$ with respect to $\eta_i$ of the map 
            
            \begin{equation*}
                \EE\left[|v^{\lambda, \eta_{i+1}}({t}_{i+1}, X_{{t}_{i+1}}) - \mathscr{V}^{\eta_{i}}(X_{{t}_{i}})|^{2}\right] \approx \frac{1}{L} \sum_{l=1}^{L} |v^{\lambda, \eta_{i+1}}({t}_{i+1}, X_{{t}_{i+1}}^{(l)}) - \mathscr{V}^{\eta_{i}}(X_{{t}_{i}}^{(l)})|^{2}.
		  \end{equation*}
          
		Update: $\eta_{i} \leftarrow \eta_{i} - \rho_{E} \nabla _{\eta_{i}}$
	   }
    Set $v^{\lambda, \eta_{i}}(t_{i}, \cdot) =   \mathscr{V}^{\eta_{i}}(\cdot) + \lambda\Phi(\lambda^{-1}(P_{t_i}(\cdot) - \mathscr{V}^{\eta_{i}}(\cdot)))$.
}
\For{each epoch $m = 1, \hdots, M$}
{
	Compute the gradient $\nabla_{\theta}$ with respect to $\theta$ of the map
    \begin{align*}
		 &\mathbb{E}\Big[ \sum_{i=0}^{N} \int_{{t}_i}^{{t}_{i+1}} (v^{\lambda, \eta_{i+1}}({t}_{i+1}, X_{{t}_{i+1}})  - v^{\lambda, \theta}(t, X_{t}))^2 \, dt  \Big] \\
         & \quad 
	\approx \frac{1}{L}\sum_{l=1}^{L} \sum_{i=0}^{N}\sum_{{t}_{i}< \hat{t}_{k} < {t}_{i+1}}(v^{\lambda, \eta_{i+1}}(t_{i+1}, X_{{t}_{i+1}}^{(l)})  - v^{\lambda, \theta}(\hat{t}_{k}, X_{\hat{t}_{k}}^{(l)}))^{2} \Delta t
	\end{align*}
	 Update: $\theta \leftarrow \theta - \rho_{G} \nabla _{\theta}$
}
{\bf Return}: The maps $(v^{\lambda, \eta_{i}})_{0\leq i \leq N}$ and $v^{\lambda, \theta}$.
\caption{Bermudan option algorithm }
\end{algorithm2e}

\subsubsection{Policy Improvement Algorithm}
We now present the policy improvement algorithm. The purpose of this section is to highlight both the convergence and the monotonicity properties of the scheme described in Subsection~\ref{PIA}. In particular, we aim to establish three key facts:

\begin{enumerate}
\item[(i)] The policy iteration algorithm converges after a finite number of iterations—specifically, no more than the number of admissible exercise dates, as shown in Theorem~\ref{thm: policy_convergence}.
\item[(ii)] The sequence of value functions is non-decreasing across iterations, that is, $V^{\lambda,n+1} \geq V^{\lambda,n}$ for all $n \geq 1$.
\item[(iii)] For a fixed number of iterations $n$, the price approximation improves as the regularisation parameter decreases: if $\lambda_1 \leq \lambda_2$ then $V^{\lambda_1,n} \geq V^{\lambda_2,n}$.
\end{enumerate}

The algorithm alternates between estimating continuation values and updating the policy, thereby ensuring monotonic improvement of the value function. Each stage involves training neural networks to approximate conditional expectations, with separate updates for policy evaluation and policy improvement.

More precisely, $v^{\eta_i,k}$ denotes the approximation of the value function at exercise date $t_i$ in the $k$-th iteration, while $v^{\theta,k}$ represents its time-continuous counterpart. At each iteration, we train a neural network on simulated sample paths of the process $X$ to approximate the continuation values $v_{t_i+}^{\lambda,n}$ and the value function between exercise dates, following the same strategy as in the TD-based BSDE solver. These approximations feed into the policy evaluation and update steps \eqref{eq: pin1}–\eqref{eq: b_Vlamb_n}.

For clarity, the pseudo-code of the full algorithm is provided in Algorithm~\ref{alg: policy_algo}.

\begin{algorithm2e}[H] \label{alg: policy_algo} 
\DontPrintSemicolon 
\SetAlgoLined 
\vspace{1mm}
{\bf Input data}: Number of training epochs $M$, exercise dates $0 = t_0 <  \ldots < t_N < t_{N+1} = T$, number of mesh time-grid $n>>N$ (time step $\Delta t$ $=$ $T/n$, the dates $\hat{t}_k=k \Delta t$, $k=0, \cdots, n$ are compatible with the exercise dates $t_i$), number of policy iterations $K > N$, learning rates $\rho_G$, $\rho_E$ for the two steps, $R$ sample paths, batch size $L << R$. Temperature parameter $\lambda>0$.  Neural networks $v^{\theta}$ and $\mathscr{V}^{\eta_{i}}$ of the value function and conditional expectations. 

{\bf Initialization}: $v^{\eta_{N+1}, k}(t_{N+1}, x) = P_T(x)$ for $k=0, \dots, K$, $v^{\theta, 0}(t_{i}, x) = \EE[P_{T}|X_{t_{i}}=x]$ and $\mathscr{V}^{\eta_{i}, 0}(t_{i}, x) = \EE[P_{T}|X_{t_{i}}=x]$ for $i=0, \hdots, n$. Generate $L$ paths $(X_{\hat{t}_{k}}^{(l)})_{0\leq k \leq  n}$, $l=1, \cdots, L$ of $X$.
\\
\For{$k = 0, \dots, K-1$}{
    \For{exercise dates $t_{i} = t_{N},\ldots, t_{0}$}
{ 
    \For{each epoch $m = 1, \hdots, M$}{
           Compute the gradient $\nabla_{\eta_{i}}$ with respect to $\eta_i$ of the map 
		  
            \begin{equation*}
                \EE\left[|v^{\eta_{i+1}, k+1}(t_{i+1}, X_{t_{i+1}}) - \mathscr{V}^{\eta_{i}, k+1}(X_{t_{i}})|^{2}\right] \approx \frac{1}{L} \sum_{l=1}^{L} |v^{\eta_{i+1}, k+1}(t_{i+1}, X_{t_{i+1}}^{(l)}) - \mathscr{V}^{\eta_{i}, k+1}(X_{t_{i}}^{(l)})|^{2}
		  \end{equation*}
		  with respect to the parameters $\eta_{i}$. Update: $\eta_{i} \leftarrow \eta_{i} - \rho_{E} \nabla _{\eta_{i}}$
	   }
    Set 
    \begin{align*}
        \pi^{k+1}_{t_i}(u) & : = \pi^{*}_{t_i}(\mathscr{V}^{\eta_{i}, k},u)  = \frac{\frac{1}{\lambda}(P_{t_{i}}-\mathscr{V}^{\eta_{i}, k})}{e^{\frac{1}{\lambda}(P_{t_{i}}-\mathscr{V}^{\eta_{i}, k})}-1} e^{\frac{1}{\lambda}(P_{t_{i}}-\mathscr{V}^{\eta_{i}, k})u}, \quad u \in [0,1], 
    \end{align*}
    and 
    \begin{equation*}
        G(t_{i}, \mathscr{V}^{\eta_{i}, k+1}, \pi_{t_{i}}^{k+1}) = \lambda \Phi\left(\frac{P_{t_{i}} - \mathscr{V}^{\eta_{i}, k}}{\lambda}\right) + (\mathscr{V}^{\eta_{i}, k} - \mathscr{V}^{\eta_{i}, k+1})\times \mu_{\pi_{
        t_{i}}^{k+1}}.
    \end{equation*}
    Set $v^{\eta_{i}, k+1}(t_{i}, \cdot) =  \mathscr{V}^{\eta_{i}, k+1}(\cdot) + G(t_{i}, \mathscr{V}^{\eta_{i}, k+1}, \pi_{t_{i}}^{k+1})$.
}
\For{each epoch $m = 1, \hdots, M$}
{
	Compute the gradient $\nabla_{\theta}$ with respect to $\theta$ of the map
    \begin{align*}
		 &\mathbb{E}\Big[ \sum_{i=0}^{N} \int_{t_i}^{t_{i+1}} (v^{\eta_{i+1}, k+1}(t_{i+1}, X_{t_{i+1}})  - v^{\theta, k+1}(t, X_{t}))^2 \, dt  \Big] \\
	   &\approx \frac{1}{L}\sum_{l=1}^{L} \sum_{i=0}^{N}\sum_{t_{i}< \hat{t}_{j} < t_{i+1}}(v^{\eta_{i+1}, k+1}(t_{i+1}, X_{t_{i+1}}^{(l)})  - v^{\theta, k+1}(\hat{t}_{j}, X_{\hat{t}_{j}}^{(l)}))^{2} \Delta t
	\end{align*}
	with respect to the parameters $\theta$. Update: $\theta \leftarrow \theta - \rho_{G} \nabla _{\theta}$
}
}
{\bf Return}: The maps $(v^{\eta_{i}, K})_{0\leq i \leq N}$ and $v^{\theta, K}$.
\caption{Policy improvement algorithm}
\end{algorithm2e}

\subsection{Numerical results}

\subsubsection{Bermudan Put option} \label{ex: policy_1d}
We implement the TD-based BSDE solver and policy improvement algorithm to compute the price of a Bermudan put option whose price at time $0$ is given by
\begin{equation*}
\sup_{\tau \in \cT(S)} \EE[e^{-r\tau}(K - X_{\tau})^{+}],
\end{equation*}

\noindent where $X$ follows a standard one-dimensional Black–Scholes dynamics 
$$
X_t = x_0 \exp\big((r-\sigma^2/2)t+\sigma W_t\big), \quad t\in [0,T].
$$

We consider the case $T=1$ and $N=3$, so that $\cT(S) = \left\{0, 0.25, 0.50, 0.75\right\} \cup \left\{1\right\}$, and take $x_0 = 100$, $r = 0.05$, $\sigma = 0.1$, and $K = 100$.

Table \ref{tab: policy_simple} reports the values of $v_0^{\lambda,n} \approx V_0$ at each iteration of the policy improvement scheme defined in \eqref{eq: b_Vlamb_n}-\eqref{eq: Gpi_def} and the TD-based BSDE solver. The initial value is given by the Black-scholes price of the European put option $v^{\lambda, 0}_{0} = \EE[P_{T}]$. At each iteration, we train $N=3$ neural networks on batches of $L=8192$ sample paths of $X$. Note that according to Corollary \ref{cor: policy_convergence}, the scheme converges after $n = N = 3$ iterations.

We observe that the value of $v_0^{\lambda, n}$ decreases as $\lambda$ increases. However, for $n \geq 1$, the monotonicity implied by the comparison theorem holds only up to $n=3$. For reference, the Black–Scholes European price of the put option is 1.928, while the Longstaff–Schwartz projection method yields an estimate of 2.311.

\begin{table}[h] 
    \centering
    \begin{tabular}{r ccccccc c}
        \toprule
            &   &      \multicolumn{4}{c}{\textbf{Policy iteration}} &  & {\textbf{TD Solver}} & \\
\cmidrule(lr){2-7} \cmidrule(lr){8-8}
        \textit{$\lambda$} & $v_{0}^{\lambda, 0}$ & $v_{0}^{\lambda, 1}$ & $v_{0}^{\lambda, 2}$ & $v_{0}^{\lambda, 3}$ & $v_{0}^{\lambda, 4}$ & $v_{0}^{\lambda, 5}$ & $v_{0}^{\lambda}$\\ \midrule

        0.1 & 1.928 & 1.639 & 1.647 & 1.649 & 1.647 & 1.649 & 1.645\\
        0.01 & 1.928 & 2.165 & 2.179 & 2.180 & 2.180 & 2.179 & 2.175\\
        0.001 & 1.928 & 2.286 & 2.298 & 2.303 & 2.302 & 2.302 & 2.299\\

        \bottomrule
    \end{tabular}
    \caption{Values of $v_0^{\lambda,n}$ obtained by the policy improvement algorithm for $\lambda = 0.1, 0.01, 0.001$ and $n=0, \cdots, 5$. Values from the TD solver $v^{\lambda}_{0}$ are also provided.}
    \label{tab: policy_simple}
\end{table}

\subsubsection{Bermudan max-call option}\label{max:call:option}
We consider the symmetric case of a Bermudan max-call option, as described in Becker \emph{et al.} \cite{BCJ2019} and references therein. In particular, we assume that the underlying assets follow a $d$-dimensional Black–Scholes model with dividends:
\begin{equation} \label{eq: GBMwD}
X_t^i = x_0^i \exp\big((r - \delta - \sigma^2/2) t + \sigma W_t^i\big), \quad i = 1, \dots, d,
\end{equation}

\noindent where $x_0^i$ denotes the initial values, $r$ the risk-free rate, $\delta$ the constant dividend yield, $\sigma$ the constant volatility, and $W=(W^1, \cdots, W^d)$ a standard $d$-dimensional Brownian motion.

We recall that the price of the Bermudan max-call option is given by 
\begin{equation*}
    \sup_{\tau \in \cT(S)} \EE\left[e^{-r\tau} \left(\max_{1 \leq i \leq d}X_{\tau}^{i} - K\right)^{+}\right],
\end{equation*}

\noindent where $\cT(S)$ denotes the set of stopping times taking values in $\left\{t_{0}, \ldots, t_{N}\right\} \cup \left\{T\right\}$, with $t_{n} = nT/(N+1)$ for $n = 0, \cdots, N$. The model parameters are $r = 0.05$, $\delta = 0.1$, $\sigma = 0.2$, $\rho = 0$, $K = 100$, $T = 3$, and $N = 8$, in line with our chosen notation. 

Following Becker \emph{et al.} \cite{BCJ2019}, each time point corresponds to an exercise date, so it is sufficient to approximate only the continuation values. In each training epoch, we simulate a batch of 8192 paths sampled from \eqref{eq: GBMwD}, and estimate the value at each exercise date using \eqref{eq: adjustment_condition_network}. Nonetheless, Algorithm \ref{alg: offalgo_2} can also compute prices between exercise dates if a finer time grid is required. From an implementation standpoint, our approach is close to the projection method of Longstaff and Schwartz. The main differences are: (i) the adjustment imposed by condition \eqref{eq: adjustment_condition_network} to account for each exercise opportunity, and (ii) the use of neural networks and the TD error to estimate continuation values.

We first illustrate the convergence of the policy improvement algorithm in dimension $d=2$ for different values of the temperature parameter $\lambda$ and initial stock price $x_0^1=x_0^2=x_0$. Table \ref{tab: policy_max_call} presents the estimated option values at each iteration of the policy improvement scheme, as defined in \eqref{eq: b_Vlamb_n} and \eqref{eq: Gpi_def}. As in the case of the Bermudan put option, convergence appears to occur at $n = N = 8$, and the results also exhibit monotonicity with respect to $\lambda$. 

\begin{table}[h]
    \centering
    \begin{tabular}{rc ccccccccccc}
        \toprule
            &         & \multicolumn{9}{c}{\textbf{Policy iteration}}\\
\cmidrule(lr){3-11} 
        \textit{$\lambda$} & $x_0$ & $v_{0}^{\lambda, 0}$ & $v_{0}^{\lambda, 1}$ & $v_{0}^{\lambda, 2}$ & $v_{0}^{\lambda, 3}$ & $v_{0}^{\lambda, 4}$ & $v_{0}^{\lambda, 5}$ & $v_{0}^{\lambda, 6}$ & $v_{0}^{\lambda, 7}$ & $v_{0}^{\lambda, 8}$ \\ \midrule

        0.1 & 90 & 6.656 & 5.653 & 5.697 & 5.700 & 5.717 & 5.700 & 5.704 & 5.706 & 5.705 \\
        0.1 & 100 & 11.193 & 11.494 & 11.577 & 11.583 & 11.584 & 11.581 & 11.586 & 11.583 & 11.582 \\
        0.1 & 110 & 16.929 & 19.126 & 19.284 & 19.283 & 19.293 & 19.298 & 19.293 & 19.302 & 19.301 \\
        \addlinespace[0.3em]

        0.01 & 90 & 6.656 & 7.578 & 7.652 & 7.667 & 7.684 & 7.680 & 7.680 & 7.679 & 7.679 \\
        0.01 & 100 & 11.193 & 13.293 & 13.497 & 13.500 & 13.501 & 13.497 & 13.501 & 13.502 & 13.502 \\
        0.01 & 110 & 16.929 & 20.661 & 21.000 & 20.966 & 20.969 & 20.967 & 20.971 & 20.976 & 20.978 \\
        \addlinespace[0.3em]

        0.001 & 90 & 6.656 & 7.921 & 8.027 & 8.029 & 8.030 & 8.030 & 8.030 & 8.029 & 8.031 \\
        0.001 & 100 & 11.193 & 13.635 & 13.870 & 13.849 & 13.849 & 13.850 & 13.848 & 13.850 & 13.848 \\
        0.001 & 110 & 16.929 & 20.913 & 21.310 & 21.312 & 21.313 & 21.308 & 21.299 & 21.301 & 21.302 \\
        \bottomrule
    \end{tabular}
    \caption{Values of $v_0^{\lambda, n}$ obtained by the policy improvement algorithm for $\lambda = 0.1, 0.01, 0.001$, $n=0, \cdots, 8$, and $x_0=90, 100, 110$.}
    \label{tab: policy_max_call}
\end{table}

Then, we present the numerical results obtained from the TD-based BSDE solver and the policy improvement algorithm (after $n=N=8$ iterations) when the dimension $d$ varies in Table \ref{tab: res_alg}. Each entry represents the estimated value function at time $t_{0}=0$ with initial prices ${x}^{i}_{0} = x_{0}$, $i=1, \cdots, d$, where $x_{0} \in \left\{90, 100, 110\right\}$. We observe that as $\lambda$ decreases, the estimates increase monotonically, in agreement with Remark \ref{rem: monotone_conv}.   

\vskip10pt
\begin{table}[h]
    \centering
    \begin{tabular}{cc ccc ccc c}
        \toprule
            &         & \multicolumn{3}{c}{\textbf{TD-based BSDE solver}} & \multicolumn{3}{c}{\textbf{Policy Improvement}} &         \\
\cmidrule(lr){3-5} \cmidrule(lr){6-8}
        \textit{d} & $x_0$ & $v_{0}^{0.1, \theta}$ & $v_{0}^{0.01, \theta}$ & $v_{0}^{0.001, \theta}$ & $v_{0}^{0.1, 8}$ & $v_{0}^{0.01, 8}$ & $v_{0}^{0.001, 8}$ & Becker \emph{et al.}\@ \cite{BCJ2019} \\ \midrule

        2 & 90 & 5.694 & 7.666 & 8.020 & 5.705 & 7.679 & 8.031 & 8.074\\
        2 & 100 & 11.580 & 13.504 & 13.849 & 11.582 & 13.502 & 13.848 & 13.899\\
        2 & 110 & 19.256 & 20.977 & 21.296 & 19.301 & 20.978 & 21.302 & 21.349\\
        \addlinespace[0.3em]

        10 & 90 & 22.900 & 25.846 & 26.298 & 23.067 & 25.996 & 26.450 & 26.240\\
        10 & 100 & 35.052 & 37.945 & 38.371 & 35.180 & 38.157 & 38.599 & 38.337\\
        10 & 110 & 47.536 & 50.463 & 50.960 & 47.816 & 50.838 & 51.268 & 50.886\\
        \addlinespace[0.3em]

        50 & 90 & 50.279 & 53.704 & 54.161 & 50.643 & 54.454 & 55.136 & 54.057\\
        50 & 100 & 65.683 & 69.301 & 69.830 & 66.867 & 70.645 & 70.725 & 69.736\\
        50 & 110 & 81.575 & 84.993 & 85.606 & 82.481 & 86.214 & 86.236 & 85.463\\
        
        \bottomrule
    \end{tabular}
    \caption{Results of the TD-based BSDE solver and the policy improvement algorithm for $\lambda = 0.1, 0.01, 0.001$.}
    \label{tab: res_alg}
\end{table}

\section{Appendix}\label{AL}

\subsection{Proofs for Bermudan options}

\subsubsection{ Proof of Lemma \ref{randomized}}\label{proof:lem:randomized}
The result clearly holds for $t = T$, so we consider the case where $t < T$. We first show that the process $V$ given by \eqref{RBSDEbermuda} is greater than the right-hand side of \eqref{RV}. Writing the dynamics of \eqref{RBSDEbermuda} forward gives
\begin{equation*}
	V_{t} = V_{0} + M_{t} - \sum_{0 \leq t_{i} < t}(P_{t_{i}} - V_{t_{i}+})^{+}. 
\end{equation*}
Let $\Gamma = \sum_{i=0}^N \Delta \Gamma_{t_i}\I_{\llb t_i,\infty \llb}$ with $\Delta \Gamma \in [0,1)$. Then by It\^o-Gal'c\v uk-Lenglart formula, we have
\begin{align*}
	\cE_{t}(-\Gamma_{-})V_{t} &= V_{0} + \int_{]0,t]}\cE_{s-}(-\Gamma_{-}) \, dM_{s} - \int_{[0,t[} V_{s}\cE_{s}(-\Gamma_{-}) \, d\Gamma_{s} \\ 
	&\hspace{1em} - \sum_{0 \leq t_{i} < t}\cE_{t_{i}}(-\Gamma_{-})(P_{t_{i}} - V_{t_{i}+})^{+} + \sum _{0 \leq s < t} \Delta^{+}\cE_{s}(-\Gamma_{-})\Delta^{+}V_{s} \\
	&= V_{0} + \int_{]0,t]}\cE_{s-}(-\Gamma_{-}) \, dM_{s} - \sum_{0 \leq t_{i} < t} V_{t_{i}}\cE_{t_{i}}(-\Gamma_{-}) \Delta\Gamma_{t_{i}} \\ 
	&\hspace{1em} - \sum_{0 \leq t_{i} < t}\cE_{t_{i}}(-\Gamma_{-})(P_{t_{i}} - V_{t_{i}+})^{+} + \sum _{0 \leq t_{i} < t} \cE_{t_{i}}(-\Gamma_{-})\Delta \Gamma_{t_{i}}(P_{t_{i}} - V_{t_{i}+})^{+} \\
	&= V_{0} + \int_{]0,t]}\cE_{s-}(-\Gamma_{-}) \, dM_{s} - \sum_{0 \leq t_{i} < t}V_{t_{i}+}\cE_{t_{i}}(-\Gamma_{-})\Delta \Gamma_{t_{i}} - \sum_{0 \leq t_{i} < t}\cE_{t_{i}}(-\Gamma_{-})(P_{t_{i}} - V_{t_{i}+})^{+},
\end{align*}
where we use the fact that $V_{t_{i}} = V_{t_{i}+} + (P_{t_{i}} - V_{t_{i}+})^{+}$ in the final equality. Hence, 
\begin{align*}
	\cE_{t}(-\Gamma_{-})V_{t} &= P_{T}\cE_{T}(-\Gamma_{-}) - \int_{]t,T]}\cE_{s-}(-\Gamma_{-}) \, dM_{s} \\
	&\hspace{1em} + \sum_{t \leq t_{i} < T}V_{t_{i}+} \cE_{t_{i}}(-\Gamma_{-}) \Delta \Gamma_{t_{i}} + \sum_{t \leq t_{i} < T} \cE_{t_{i}}(-\Gamma_{-})(P_{t_{i}} - V_{t_{i}+})^{+} \\
	&\geq P_{T}\cE_{T}(-\Gamma_{-}) - \int_{]t,T]}\cE_{s-}(-\Gamma_{-}) \, dM_{s} \\
	&\hspace{1em} + \sum_{t \leq t_{i} < T}V_{t_{i}+} \cE_{t_{i}}(-\Gamma_{-}) \Delta \Gamma_{t_{i}} + \sum_{t \leq t_{i} < T} \cE_{t_{i}}(-\Gamma_{-})(P_{t_{i}} - V_{t_{i}+})^{+}\Delta \Gamma_{t_{i}} \\
	&= P_{T}\cE_{T}(-\Gamma_{-}) - \int_{]t,T]}\cE_{s-}(-\Gamma_{-}) \, dM_{s} + \sum_{t \leq t_{i} < T}V_{t_{i}} \cE_{t_{i}}(-\Gamma_{-}) \Delta \Gamma_{t_{i}} \\
	&\geq P_{T}\cE_{T}(-\Gamma_{-}) - \int_{]t,T]}\cE_{s-}(-\Gamma_{-}) \, dM_{s} + \int_{[t,T[} P_{s} \cE_{s}(-\Gamma_{-}) \, d\Gamma_{s}
\end{align*}
where we again use the fact that $V_{t_{i}} = V_{t_{i}+} + (P_{t_{i}} - V_{t_{i}+})^{+}$ and $V_{t_{i}} \geq P_{t_{i}}$ for all $t_{i} \in S$. Therefore, 
\begin{equation*}
	V_{t} \geq \esssup_{\Gamma \in \Pi} \EE[P_{T}\cE_{t,T}(-\Gamma_{-}) + \int_{[t,T[} P_{s}\cE_{t,s}(-\Gamma_{-}) \, d\Gamma_{s} | \cF_{t}].
\end{equation*}
To obtain the reverse inequality, we consider
\begin{align*}
\Pi_n = \left\{\Gamma \in \Pi: \Delta \Gamma_{t_i} \in [0,1-1/n],\,\,i = 0,1,\dots, N\right\}  \subset \Pi
\end{align*}
and we show that the process $V$ given in \eqref{RBSDEbermuda} is less than the right hand side of \eqref{RV}. To this end we consider a sequence of BSDEs $(\widetilde{Y}^{n})_{n}$ given by 
\begin{equation*}
	\widetilde{Y}^{n} = P_{T} - (\widetilde{M}^{n}_{T} - \widetilde{M}^{n}_{t}) + \sum_{t \leq t_{i} < T} \left(1 - \frac{1}{n} \right)(P_{t_{i}} - \widetilde{Y}^{n}_{t_{i}+})^{+}. 
\end{equation*}
By comparison theorem we have $\widetilde{Y}^{n} \leq \widetilde{Y}^{n+1}$ and that $\widetilde{Y}^{n}$ converges to $V$, where $V$ is the BSDE \eqref{RBSDEbermuda}. Now consider the function
\begin{equation}
	g(x) = \left(1 - \frac{1}{n}\right) \I_{\{P - x > 0\}} + \epsilon \left(\I_{\{-1 \leq P - x \leq 0\}} + |P-x|^{-1}\I_{\{P-x < -1\}}\right),
\end{equation}
which is Lipschitz continuous and satisfies the inequality $0 \leq (1 - \frac{1}{n})(P-x)^{+} - (P - x)g(x) \leq \epsilon$. Thus by the comparison theorem, we have that 
\begin{align*}
	\widetilde{Y}^{n}_{t} &= P_{T} - (\widetilde{M}^{n}_{T} - \widetilde{M}^{n}_{t}) + \sum_{t \leq t_{i} < T} \left(1 - \frac{1}{n} \right)(P_{t_{i}} - \widetilde{Y}^{n}_{t_{i}+})^{+} \\
	&\leq P_{T} - (\widetilde{M}^{n, \epsilon}_{T} - \widetilde{M}^{n, \epsilon}_{t}) + \sum_{t \leq t_{i} < T} (P_{t_{i}} - \widetilde{Y}^{n, \epsilon}_{t_{i}+})g(\widetilde{Y}^{n, \epsilon}_{t_{i}+}) + \epsilon =: \widetilde{Y}^{n, \epsilon}_{t}. 
\end{align*}
Next, consider the BSDE
\begin{equation*}
	V_{t}^{n, \epsilon} = P_{T} - (M_{T}^{n,\epsilon} - M_{t}^{n, \epsilon}) + \sum_{t \leq t_{i} < T}(P_{t_{i}} - V_{t_{i}+}^{n, \epsilon})g(\widetilde{Y}_{t_{i}+}^{n, \epsilon}) \leq V_{t}^{n} := \esssup_{\varphi \in \Pi_{n}}V_{t}^{\varphi}, 
\end{equation*}
where 
\begin{equation*}
    V_{t}^{\varphi} := P_{T} - (M_{T}^{\varphi} - M_{t}^{\varphi}) + \sum_{t \leq t_{i} < T}(P_{t_{i}} - V_{t_{i}+}^{\varphi}) \Delta\varphi_{t_{i}}.
\end{equation*}
Moreover, by applying the It\^o-Gal'c\v uk-Lenglart formula to $\cE(-\varphi_{-})V^{\varphi}$ and the identity $V_{t_{i}}^{\varphi} = V_{t_{i+}}^{\varphi} + (P_{t_{i}} - V_{t_{i}+}^{\varphi}) \Delta \varphi_{t_{i}}$, we observe that 
\begin{equation*}
	V^\varphi_t = \mathbb{E}[P_T\mathcal{E}_{t,T}(-\varphi_-) + \int_{[t,T[} P_s \mathcal{E}_{t,s}(-\varphi_-) d\varphi_s|\cF_t]. 
\end{equation*}

Now, since that $g$ is positive and $\widetilde{Y}^{n, \epsilon} - V^{n, \epsilon} \geq 0$ due to comparison theorem, it follows that 
\begin{align*}
	\widetilde{Y}_{t}^{n, \epsilon} - V_{t}^{n, \epsilon} &= (M_{T}^{n, \epsilon} - M_{t}^{n, \epsilon}) - (\widetilde{M}_{T}^{n, \epsilon} - \widetilde{M}_{t}^{n, \epsilon}) + \sum_{t \leq t_{i} < T} \left\{(P_{t_{i}} - \widetilde{Y}^{n, \epsilon}_{t_{i}+})g(\widetilde{Y}^{n, \epsilon}_{t_{i}+}) + \epsilon - (P_{t_{i}} - V^{n, \epsilon}_{t_{i}+})g(\widetilde{Y}^{n, \epsilon}_{t_{i}+})\right\} \\
	&= (M_{T}^{n, \epsilon} - M_{t}^{n, \epsilon}) - (\widetilde{M}_{T}^{n, \epsilon} - \widetilde{M}_{t}^{n, \epsilon}) + \sum_{t \leq t_{i} < T} \left\{( V^{n, \epsilon}_{t_{i}+} - \widetilde{Y}^{n, \epsilon}_{t_{i}+})g(\widetilde{Y}^{n, \epsilon}_{t_{i}+}) + \epsilon \right\} \\
	&\leq (M_{T}^{n, \epsilon} - M_{t}^{n, \epsilon}) - (\widetilde{M}_{T}^{n, \epsilon} - \widetilde{M}_{t}^{n, \epsilon}) + \sum_{t \leq t_{i} < T} \epsilon.
\end{align*}
Upon taking conditional expectations and noting that the above inequality holds for any $\epsilon > 0$ gives
\begin{equation*}
	\widetilde{Y}_{t}^{n, \epsilon} \leq V_{t}^{n, \epsilon} \leq V_{t}^{n}. 
\end{equation*}
But we have $\widetilde{Y}_{t}^{n} \leq V_{t}^{n} \leq \esssup_{\Gamma \in \Pi} \mathbb{E}[P_T\mathcal{E}_{t,T}(-\Gamma_-) + \int_{[t,T[} P_s \mathcal{E}_{t,s}(-\Gamma_-) d\Gamma_s|\cF_t]$. Thus, taking limit as $n \rightarrow \infty$ yields
\begin{equation*}
	V_{t} \leq \esssup_{\Gamma \in \Pi} \mathbb{E}[P_T\mathcal{E}_{t,T}(-\Gamma_-) + \int_{[t,T[} P_s \mathcal{E}_{t,s}(-\Gamma_-) d\Gamma_s|\cF_t].
\end{equation*}

\subsection{Proofs for Bermudan game options}
We first give some auxiliary results that will be useful for the proof of Lemma \ref{prop: game_est}.
\bl \label{lem: game_est_G_H_diff}
    For $t \in [0,T)$, we have
    \begin{equation} \label{eq: game_est_G_H_diff_eq1}
        V_{t}^{n+1, n+1} - V_{t}^{n, n+1} \leq \sum_{t \leq t_{i} < T} \EE[ (V_{t_{i}+}^{n,n} - V_{t_{i}+}^{n,n+1}) + (V_{t_{i}+}^{n,n} - V_{t_{i}+}^{n-1, n}) \, | \, \cF_{t}],
    \end{equation}
    and
    \begin{equation} \label{eq: game_est_G_H_diff_eq2}
        V_{t}^{n, n} - V_{t}^{n, n+1} \leq \sum_{t \leq t_{i} < T} \EE[(V_{t_{i}+}^{n, n} - V_{t_{i}+}^{n-1, n}) + (V_{t_{i}+}^{n-1,n-1} - V_{t_{i}+}^{n-1, n}) \, | \, \cF_{t}]
    \end{equation}
    and, for the right continuous modifications, we have
    \begin{equation} \label{eq: game_est_G_H_diff_eq3}
        V_{t+}^{n+1, n+1} - V_{t+}^{n, n+1} \leq \sum_{t < t_{i} < T} \EE[ (V_{t_{i}+}^{n,n} - V_{t_{i}+}^{n,n+1}) + (V_{t_{i}+}^{n,n} - V_{t_{i}+}^{n-1, n}) \, | \, \cF_{t}],
    \end{equation}
    and
    \begin{equation} \label{eq: game_est_G_H_diff_eq4}
        V_{t+}^{n, n} - V_{t+}^{n, n+1} \leq \sum_{t < t_{i} < T} \EE[(V_{t_{i}+}^{n, n} - V_{t_{i}+}^{n-1, n}) + (V_{t_{i}+}^{n-1,n-1} - V_{t_{i}+}^{n-1, n}) \, | \, \cF_{t}].
    \end{equation}
\el
\begin{proof}
\emph{Step 1: }    We first show \eqref{eq: game_est_G_H_diff_eq1}. Note that
    \begin{align*}
        V_{t}^{n+1, n+1} - V_{t}^{n, n+1}&= -\left[ \left(M_{T}^{n+1, n+1} + M_{T}^{n, n+1}\right) - \left(M_{t}^{n+1, n+1} + M_{t}^{n, n+1}\right)\right] \\
    &\hspace{1em} + \sum_{t \leq t_{i} < T} G(t_{i}, V_{t_{i}+}^{n+1, n+1}, \underline{\pi}^{n+1}_{t_{i}}) - G(t_{i}, V_{t_{i}+}^{n, n}, \underline{\pi}^{n}_{t_{i}}).
\end{align*}
Since $x \mapsto \Phi(x)$ is increasing and $\Phi$ is Lipschitz continuous we observe that the summand satisfies
\begin{align*}
     & G(t_{i}, V_{t_{i}+}^{n+1, n+1}, \underline{\pi}^{n+1}_{t_{i}}) - G(t_{i}, V_{t_{i}+}^{n, n}, \underline{\pi}^{n}_{t_{i}}) \\
     &= \lambda \Phi \left(\frac{P_{t_{i}} - V_{t_{i}+}^{n, n+1}}{\lambda}\right) - \lambda \Phi \left(\frac{P_{t_{i}} - V_{t_{i}+}^{n-1, n}}{\lambda}\right) + (V_{t_{i}+}^{n, n+1} - V_{t_{i}+}^{n+1, n+1})\mu_{\underline{\pi}_{t_{i}}^{n+1}} - (V_{t_{i}+}^{n-1, n} - V_{t_{i}+}^{n, n})\mu_{\underline{\pi}_{t_{i}}^{n}} \\
     &\leq \lambda \Phi \left(\frac{P_{t_{i}} - V_{t_{i}+}^{n, n+1}}{\lambda}\right) - \lambda \Phi \left(\frac{P_{t_{i}} - V_{t_{i}+}^{n, n}}{\lambda}\right) + (V_{t_{i}+}^{n,n} - V_{t_{i}+}^{n-1, n}) \\
     &\leq (V_{t_{i}+}^{n,n} - V_{t_{i}+}^{n,n+1}) + (V_{t_{i}+}^{n,n} - V_{t_{i}+}^{n-1, n}).
\end{align*}
It follows by taking conditional expectation with respect to $\cF_{t}$ that
\begin{equation*}
    V_{t}^{n+1, n+1} - V_{t}^{n, n+1} \leq \sum_{t \leq t_{i} < T} \EE[ (V_{t_{i}+}^{n,n} - V_{t_{i}+}^{n,n+1}) + (V_{t_{i}+}^{n,n} - V_{t_{i}+}^{n-1, n}) \, | \, \cF_{t}].
\end{equation*}
By repeating the above for the right-continuous modification, we also obtain \eqref{eq: game_est_G_H_diff_eq3}.\\

\noindent \emph{Step 2:} We now show \eqref{eq: game_est_G_H_diff_eq4}. As above, we note that
\begin{align*}
    V_{t+}^{n, n} - V_{t+}^{n, n+1} &= -\left[ \left(M_{T}^{n, n} + M_{T}^{n, n+1}\right) - \left(M_{t}^{n, n} + M_{t}^{n, n+1}\right)\right] \\
    &\hspace{1em} + \sum_{t < t_{i} < T} H(t_{i}, V_{t_{i}+}^{n, n+1}, \overline{\pi}^{n+1}_{t_{i}}) - H(t_{i}, V_{t_{i}+}^{n-1, n}, \overline{\pi}^{n}_{t_{i}}),
\end{align*}
and, since $x \mapsto \Phi(x)$ is increasing and $\Phi$ is Lipschitz continuous, the summand satisfies 
\begin{align*}
& H(t_{i}, V_{t_{i}+}^{n, n+1}, \overline{\pi}^{n+1}_{t_{i}}) -H(t_{i}, V_{t_{i}+}^{n-1, n}, \overline{\pi}^{n}_{t_{i}}) \\
& = \lambda \Phi\left(\frac{{V_{t_{i}+}^{n, n} - R_{t_{i}}}}{\lambda}\right) - \lambda \Phi\left(\frac{{V_{t_{i}+}^{n-1, n-1}  - R_{t_{i}}}}{\lambda}\right) + (V_{t_{i}+}^{n, n+1} - V_{t_{i}+}^{n, n}) \mu_{\overline{\pi}^{n+1}} - (V_{t_{i}+}^{n-1,n} - V_{t_{i}+}^{n-1, n-1}) \mu_{\overline{\pi}^{n}} \\
& \leq (V_{t_{i}+}^{n, n} - V_{t_{i}+}^{n-1, n}) + (V_{t_{i}+}^{n-1,n-1} - V_{t_{i}+}^{n-1, n}). 
\end{align*}
Then, taking conditional expectation with respect to $\cF_{t}$ gives
\begin{equation*}
    V_{t+}^{n, n} - V_{t+}^{n, n+1} \leq \sum_{t < t_{i} < T} \EE[(V_{t_{i}+}^{n, n} - V_{t_{i}+}^{n-1, n}) + (V_{t_{i}+}^{n-1,n-1} - V_{t_{i}+}^{n-1, n}) \, | \, \cF_{t}].
\end{equation*}
An analogous argument also shows \eqref{eq: game_est_G_H_diff_eq2}.
\end{proof}

\begin{proof}[{\bf Proof of Lemma \ref{prop: game_est}}]
We will proceed using an induction argument. From Lemma \ref{lem: game_est_G_H_diff}, we get
    \begin{align*}
        V_{t}^{n+1, n+1} - V_{t}^{n, n+1} &\leq \sum_{t \leq t_{i} < T} \EE[ (V_{t_{i}+}^{n,n} - V_{t_{i}+}^{n,n+1}) + (V_{t_{i}+}^{n,n} - V_{t_{i}+}^{n-1, n}) \, | \, \cF_{t}] \\ 
        &\leq \sum_{t \leq t_{i} < T} \sum_{t_{i} < t_{i}^{1} < T} \EE[ (V_{t_{i}^{1}+}^{n,n} - V_{t_{i}^{1}+}^{n-1,n}) + (V_{t_{i}^{1}+}^{n-1,n-1} - V_{t_{i}^{1}+}^{n-1, n}) \, | \, \cF_{t}] \\
        &\hspace{1em} + \sum_{t \leq t_{i} < T} \sum_{t_{i} < t_{i}^{1} < T} \EE[ (V_{t_{i}^{1}+}^{n-1,n-1} - V_{t_{i}^{1}+}^{n-1,n}) + (V_{t_{i}^{1}+}^{n-1,n-1} - V_{t_{i}^{1}+}^{n-2, n-1}) \, | \, \cF_{t}] \\
        &= \sum_{t \leq t_{i} < T} \sum_{t_{i} < t_{i}^{1} < T} \EE[ V_{t_{i}^{1}+}^{n,n} - V_{t_{i}^{1}+}^{n-1,n} \, | \, \cF_{t}]   + 2 \sum_{t \leq t_{i} < T} \sum_{t_{i} < t_{i}^{1} < T} \EE[ V_{t_{i}^{1}+}^{n-1,n-1} - V_{t_{i}^{1}+}^{n-1, n} \, | \, \cF_{t}] \\
        &\hspace{1em} + \sum_{t \leq t_{i} < T} \sum_{t_{i} < t_{i}^{1} < T} \EE[V_{t_{i}^{1}+}^{n-1,n-1} - V_{t_{i}^{1}+}^{n-2, n-1} \, | \, \cF_{t}],
    \end{align*}
where the second inequality follows from the tower property. Now, let us assume that
    \begin{align*}
    V_{t}^{n+1, n+1} - V_{t}^{n, n+1} 
    &\leq \sum_{t \leq t_{i} < t_{i}^{1} < \hdots < t_{i}^{2k-1} < T} \sum_{j=0}^{k}\binom{2k}{2j}\EE[V_{t_{i}^{2k-1}+}^{n+1-2k+j, n+1-2k+j} - V_{t_{i}^{2k-1}+}^{n-2k+j, n+1-2k+j}\, | \, \cF_{t}] \\
    &\hspace{1em} + \sum_{t \leq t_{i} < t_{i}^{1} < \hdots < t_{i}^{2k-1} < T} \sum_{j=0}^{k-1}\binom{2k}{2j+1}\EE[V_{t_{i}^{2k-1}+}^{n+1-2k+j, n+1-2k+j} - V_{t_{i}^{2k-1}+}^{n+1-2k+j, n+2-2k+j}\, | \, \cF_{t}].
    \end{align*}
    Then, by Lemma \ref{lem: game_est_G_H_diff} applied to each summand and the tower property, we get
    \begin{align*}
    V_{t}^{n+1, n+1} - & V_{t}^{n, n+1} 
    \leq  \sum_{t \leq t_{i} < t_{i}^{1} < \hdots < t_{i}^{2k-1} < T} \sum_{j=0}^{k}\binom{2k}{2j} \sum_{t_{i}^{2k-1}<t_{i}^{2k} < T}\EE[V_{t_{i}^{2k}+}^{n-2k+j, n-2k+j} - V_{t_{i}^{2k}+}^{n-2k+j, n+1-2k+j}\, | \, \cF_{t}] \\
    &\hspace{1em} + \sum_{t \leq t_{i} < t_{i}^{1} < \hdots < t_{i}^{2k-1} < T} \sum_{j=0}^{k}\binom{2k}{2j} \sum_{t_{i}^{2k-1}<t_{i}^{2k} < T}\EE[V_{t_{i}^{2k}+}^{n-2k+j, n-2k+j} - V_{t_{i}^{2k}+}^{n-1-2k+j, n-2k+j}\, | \, \cF_{t}] \\
    &\hspace{1em} + \sum_{t \leq t_{i} < t_{i}^{1} < \hdots < t_{i}^{2k-1} < T} \sum_{j=0}^{k-1}\binom{2k}{2j+1} \sum_{t_{i}^{2k-1}<t_{i}^{2k} < T}\EE[V_{t_{i}^{2k}+}^{n+1-2k+j, n+1-2k+j} - V_{t_{i}^{2k}+}^{n-2k+j, n+1-2k+j}\, | \, \cF_{t}] \\
    &\hspace{1em} + \sum_{t \leq t_{i} < t_{i}^{1} < \hdots < t_{i}^{2k-1} < T} \sum_{j=0}^{k-1}\binom{2k}{2j+1} \sum_{t_{i}^{2k-1}<t_{i}^{2k} < T}\EE[V_{t_{i}^{2k}+}^{n-2k+j, n-2k+j} - V_{t_{i}^{2k}+}^{n-2k+j, n+1-2k+j}\, | \, \cF_{t}].
    \end{align*}
    Then, by grouping the appropriate terms, we see that
    \begin{align*}
    V_{t}^{n+1, n+1} -&  V_{t}^{n, n+1} 
    \leq \sum_{t \leq t_{i} < t_{i}^{1} < \hdots < t_{i}^{2k} < T}\EE[V_{t_{i}^{2k}+}^{n-k, n-k} - V_{t_{i}^{2k}+}^{n-k, n+1-k}\, | \, \cF_{t}] \\ 
    &\hspace{1em} + \sum_{t \leq t_{i} < t_{i}^{1} < \hdots < t_{i}^{2k} < T} \sum_{j=0}^{k-1}\left[\binom{2k}{2j} + \binom{2k}{2j+1}\right]\EE[V_{t_{i}^{2k}+}^{n-2k+j, n-2k+j} - V_{t_{i}^{2k}+}^{n-2k+j, n+1-2k+j}\, | \, \cF_{t}]  \\
    &\hspace{1em }+ \sum_{t \leq t_{i} < t_{i}^{1} < \hdots < t_{i}^{2k} < T}\EE[V_{t_{i}^{2k}+}^{n-k, n-k} - V_{t_{i}^{2k}+}^{n-1-k, n-k}\, | \, \cF_{t}] \\
    &\hspace{1em} + \sum_{t \leq t_{i} < t_{i}^{1} < \hdots < t_{i}^{2k} < T} \sum_{j=0}^{k-1}\binom{2k}{2j} \EE[V_{t_{i}^{2k}+}^{n-2k+j, n-2k+j} - V_{t_{i}^{2k}+}^{n-1-2k+j, n-2k+j}\, | \, \cF_{t}] \\
    &\hspace{1em} + \sum_{t \leq t_{i} < t_{i}^{1} < \hdots < t_{i}^{2k} < T} \sum_{j=0}^{k-1}\binom{2k}{2j+1} \EE[V_{t_{i}^{2k}+}^{n+1-2k+j, n+1-2k+j} - V_{t_{i}^{2k}+}^{n-2k+j, n+1-2k+j}\, | \, \cF_{t}],
    \end{align*}
    whereby using the binomial identity and a change of index in the final summation, that is $j \leftarrow j + 1$, we obtain the following
    \begin{align*}
    & V_{t}^{n+1, n+1} - V_{t}^{n, n+1}\\
    &\leq \sum_{t \leq t_{i} < t_{i}^{1} < \hdots < t_{i}^{2k} < T} \sum_{j=0}^{k}\binom{2k+1}{2j+1}\EE[V_{t_{i}^{2k}+}^{n-2k+j, n-2k+j} - V_{t_{i}^{2k}+}^{n-2k+j, n+1-2k+j}\, | \, \cF_{t}] \\
    &\hspace{1em} + \sum_{t \leq t_{i} < t_{i}^{1} < \hdots < t_{i}^{2k} < T} \EE[V_{t_{i}^{2k}+}^{n-2k, n-2k} - V_{t_{i}^{2k}+}^{n-1-2k, n-2k}\, | \, \cF_{t}]\\
    &\hspace{1em} + \Bigg\{ \sum_{t \leq t_{i} < t_{i}^{1} < \hdots < t_{i}^{2k} < T} \sum_{j=1}^{k-1}\binom{2k}{2j} \EE[V_{t_{i}^{2k}+}^{n-2k+j, n-2k+j} - V_{t_{i}^{2k}+}^{n-1-2k+j, n-2k+j}\, | \, \cF_{t}] \\
    &\hspace{1em} + \sum_{t \leq t_{i} < t_{i}^{1} < \hdots < t_{i}^{2k} < T} \sum_{j=1}^{k-1}\binom{2k}{2j-1} \EE[V_{t_{i}^{2k}+}^{n-2k+j, n-2k+j} - V_{t_{i}^{2k}+}^{n-1-2k+j, n-2k+j}\, | \, \cF_{t}] \Bigg \}\\
    &\hspace{1em} + \Bigg\{\sum_{t \leq t_{i} < t_{i}^{1} < \hdots < t_{i}^{2k} < T} \binom{2k}{2k-1} \EE[V_{t_{i}^{2k}+}^{n-k, n-k} - V_{t_{i}^{2k}+}^{n-1-k, n-k}\, | \, \cF_{t}] \\
    &\hspace{1em} +  \sum_{t \leq t_{i} < t_{i}^{1} < \hdots < t_{i}^{2k} < T}\EE[V_{t_{i}^{2k}+}^{n-k, n-k} - V_{t_{i}^{2k}+}^{n-1-k, n-k}\, | \, \cF_{t}] \Bigg\}.
\end{align*}
Again, by applying the binomial identity, we obtain
\begin{align*}
    & V_{t}^{n+1, n+1} - V_{t}^{n, n+1} \\
    & \leq \sum_{t \leq t_{i} < t_{i}^{1} < \hdots < t_{i}^{2k} < T} \sum_{j=0}^{k}\binom{2k+1}{2j+1}\EE[V_{t_{i}^{2k}+}^{n-2k+j, n-2k+j} - V_{t_{i}^{2k}+}^{n-2k+j, n+1-2k+j}\, | \, \cF_{t}] \\
     &  \hspace{1em} + \sum_{t \leq t_{i} < t_{i}^{1} < \hdots < t_{i}^{2k} < T} \sum_{j=1}^{k-1}\binom{2k+1}{2j} \EE[V_{t_{i}^{2k}+}^{n-2k+j, n-2k+j} - V_{t_{i}^{2k}+}^{n-1-2k+j, n-2k+j}\, | \, \cF_{t}] \\
    &   \hspace{1em} + \sum_{t \leq t_{i} < t_{i}^{1} < \hdots < t_{i}^{2k} < T} \EE[V_{t_{i}^{2k}+}^{n-2k, n-2k} - V_{t_{i}^{2k}+}^{n-1-2k, n-2k}\, | \, \cF_{t}] \\
    &\hspace{1em} + \sum_{t \leq t_{i} < t_{i}^{1} < \hdots < t_{i}^{2k} < T} \binom{2k+1}{2k} \EE[V_{t_{i}^{2k}+}^{n-k, n-k} - V_{t_{i}^{2k}+}^{n-1-k, n-k}\, | \, \cF_{t}],
\end{align*}
and it follows that 
\begin{align} \label{eq: game_est_eq1}
    & V_{t}^{n+1, n+1} - V_{t}^{n, n+1} \nonumber \\
    &\leq \sum_{t \leq t_{i} < t_{i}^{1} < \hdots < t_{i}^{2k} < T} \sum_{j=0}^{k}\binom{2k+1}{2j+1}\EE[V_{t_{i}^{2k}+}^{n-2k+j, n-2k+j} - V_{t_{i}^{2k}+}^{n-2k+j, n+1-2k+j}\, | \, \cF_{t}]  \nonumber \\
    &\hspace{1em} + \sum_{t \leq t_{i} < t_{i}^{1} < \hdots < t_{i}^{2k} < T} \sum_{j=0}^{k}\binom{2k+1}{2j} \EE[V_{t_{i}^{2k}+}^{n-2k+j, n-2k+j} - V_{t_{i}^{2k}+}^{n-1-2k+j, n-2k+j}\, | \, \cF_{t}].
\end{align}
Repeating the above arguments starting from \eqref{eq: game_est_eq1} and, by another application of Lemma \ref{lem: game_est_G_H_diff}, we obtain
\begin{align*}
    &V_{t}^{n+1, n+1} - V_{t}^{n, n+1} \\
    &\leq \sum_{t \leq t_{i} < t_{i}^{1} < \hdots < t_{i}^{2k} < T} \sum_{j=0}^{k}\binom{2k+1}{2j} \sum_{t_{i}^{2k} < t_{i}^{2k+1} < T}\EE[V_{t_{i}^{2k+1}+}^{n-1-2k+j, n-1-2k+j} - V_{t_{i}^{2k+1}+}^{n-1-2k+j, n-2k+j}\, | \, \cF_{t}] \\
    &\hspace{1em} + \sum_{t \leq t_{i} < t_{i}^{1} < \hdots < t_{i}^{2k} < T} \sum_{j=0}^{k}\binom{2k+1}{2j} \sum_{t_{i}^{2k} < t_{i}^{2k+1} < T}\EE[V_{t_{i}^{2k+1}+}^{n-1-2k+j, n-1-2k+j} - V_{t_{i}^{2k+1}+}^{n-2-2k+j, n-1-2k+j}\, | \, \cF_{t}] \\
    &\hspace{1em} + \sum_{t \leq t_{i} < t_{i}^{1} < \hdots < t_{i}^{2k} < T} \sum_{j=0}^{k}\binom{2k+1}{2j+1} \sum_{t_{i}^{2k} < t_{i}^{2k+1} < T } \EE[V_{t_{i}^{2k+1}+}^{n-2k+j, n-2k+j} - V_{t_{i}^{2k+1}+}^{n-1-2k+j, n-2k+j}\, | \, \cF_{t}] \\
    &\hspace{1em} + \sum_{t \leq t_{i} < t_{i}^{1} < \hdots < t_{i}^{2k} < T} \sum_{j=0}^{k}\binom{2k+1}{2j+1} \sum_{t_{i}^{2k} < t_{i}^{2k+1} < T } \EE[V_{t_{i}^{2k+1}+}^{n-1-2k+j, n-1-2k+j} - V_{t_{i}^{2k+1}+}^{n-1-2k+j, n-2k+j}\, | \, \cF_{t}] \\
    &= \sum_{t \leq t_{i} < t_{i}^{1} < \hdots < t_{i}^{2k+1} < T} \sum_{j=0}^{k}\binom{2k+2}{2j+1} \EE[V_{t_{i}^{2k+1}+}^{n-1-2k+j, n-1-2k+j} - V_{t_{i}^{2k+1}+}^{n-1-2k+j, n-2k+j}\, | \, \cF_{t}] \\
    &\hspace{1em} + \sum_{t \leq t_{i} < t_{i}^{1} < \hdots < t_{i}^{2k+1} < T} \sum_{j=1}^{k}\binom{2k+1}{2j} \EE[V_{t_{i}^{2k+1}+}^{n-1-2k+j, n-1-2k+j} - V_{t_{i}^{2k+1}+}^{n-2-2k+j, n-1-2k+j}\, | \, \cF_{t}] \\
    &\hspace{1em} + \sum_{t \leq t_{i} < t_{i}^{1} < \hdots < t_{i}^{2k+1} < T} \sum_{j=1}^{k}\binom{2k+1}{2j-1} \EE[V_{t_{i}^{2k+1}+}^{n-1-2k+j, n-1-2k+j} - V_{t_{i}^{2k+1}+}^{n-2-2k+j, n-1-2k+j}\, | \, \cF_{t}] \\
    &\hspace{1em} + \sum_{t \leq t_{i} < t_{i}^{1} < \hdots < t_{i}^{2k+1} < T}  \EE[V_{t_{i}^{2k+1}+}^{n-k, n-k} - V_{t_{i}^{2k+1}+}^{n-1-k, n-k}\, | \, \cF_{t}]  \\
    &\hspace{1em} + \sum_{t \leq t_{i} < t_{i}^{1} < \hdots < t_{i}^{2k+1} < T}\EE[V_{t_{i}^{2k+1}+}^{n-1-2k, n-1-2k} - V_{t_{i}^{2k+1}+}^{n-2-2k, n-1-2k}\, | \, \cF_{t}] \\
    &= \sum_{t \leq t_{i} < t_{i}^{1} < \hdots < t_{i}^{2k+1} < T} \sum_{j=0}^{k+1}\binom{2k+2}{2j} \EE[V_{t_{i}^{2k+1}+}^{n-1-2k+j, n-1-2k+j} - V_{t_{i}^{2k+1}+}^{n-2-2k+j, n-1-2k+j}\, | \, \cF_{t}] \\ 
    &\hspace{1em} + \sum_{t \leq t_{i} < t_{i}^{1} < \hdots < t_{i}^{2k+1} < T} \sum_{j=0}^{k}\binom{2k+2}{2j+1} \EE[V_{t_{i}^{2k+1}+}^{n-1-2k+j, n-1-2k+j} - V_{t_{i}^{2k+1}+}^{n-1-2k+j, n-2k+j}\, | \, \cF_{t}].
\end{align*}
\vskip-10pt
\end{proof}

\subsection{Proofs of some technical results}\label{sec:technical:results}

\bl\label{lipphi}
The function $\Phi$ defined in \eqref{eq: Psi_Phi_def} satisfies $0 \leq \Phi'(x)\leq 1$. That is, $\Phi$ is Lipschitz continuous. 
\el

\begin{proof}
We first note for all $x \in \RR \backslash \{0\}$ that
\begin{gather*}
\Phi'(x)  = \frac{xe^x - (e^x-1)}{x(e^x-1)}, \quad \text{and } \quad \Phi''(x) =  \frac{e^{2x} - (2+x^2)e^x + 1}{x^2(e^x-1)^2}.
\end{gather*}
We can then deduce
\begin{align*}
\lim_{x\rightarrow +\infty} \Phi'(x) & = \lim_{x\rightarrow \infty} \frac{1 - (x^{-1}-(xe^{x})^{-1})}{(1-e^{-x})} = 1,\\
\lim_{x\rightarrow -\infty} \Phi'(x) & = \lim_{|x|\rightarrow \infty} \frac{e^{-|x|} - (|x|^{-1}e^{-|x|}-|x|^{-1})}{(e^{-|x|}-1)} = 0.
\end{align*}
By the standard inequality 
\begin{equation} \label{eq: cosh_ineq}
    \frac{e^{x} + e^{-x}}{2} = \cosh(x) \geq 1 + \frac{x^{2}}{2} \quad \forall x \in \RR,
\end{equation}
with equality if and only if $x=0$, we have $e^{2x} - (2+x^2)e^x + 1 \geq 0$. Hence, it follows that $\Phi''(x) > 0$ for all $x \in \RR \backslash \{0\}$. Lastly, by l'H\^opital's rule, we deduce that $\Phi'(0) = \frac{1}{2}$ and this proves the claim for all $x \in \RR$. 
\end{proof}

\bl  \label{lem: psi_cdf}
The function $x \mapsto \Psi(x)$ is a cumulative distribution function on $\mathbb{R}$.
\el 
\begin{proof}

\emph{Step 1:} We first observe that 
$$
\lim_{x \rightarrow -\infty}\Psi(x)  = \lim_{x \rightarrow -\infty}\frac{1}{x}\ln \left( \frac{e^{x}-1}{x} \right) = 0 \quad \mbox{ and } \quad \lim_{x \rightarrow +\infty}\Psi(x)    = \lim_{x \rightarrow \infty}\frac{1}{x}\ln \left( \frac{e^{x}-1}{x} \right) = 1.
$$

\noindent \emph{Step 2:} 
We now show that $\Psi$ is increasing. To this end, for $x \in \RR \backslash \{0\}$, consider 
\begin{equation*}
\Psi(x) = \frac{1}{x}\Phi(x),  \quad \text{and }\quad \Psi ' (x) = \frac{x \Phi ' (x) - \Phi(x)}{x^{2}}.
\end{equation*}

Now let $g(x) := x \Phi ' (x) - \Phi(x)$ and its derivative 
\begin{equation*}
    g'(x) = x \Phi''(x) = \frac{e^{2x} - (2+x^{2})e^{x} + 1}{x(e^{x}-1)^{2}}.
\end{equation*}
But by \eqref{eq: cosh_ineq}, it follows that $g'(x) < 0$ for $x < 0$ and $g'(x) > 0$ for $x > 0$. On the other hand, by l'H\^opital's rule we have $\lim_{x \rightarrow 0}g(x) = 0$,
which implies $g(x) > 0$, and therefore $\Psi'(x) > 0$ for all $x \in \RR \backslash \{0\}$. We conclude by noticing that $\Psi$ is continuous on $\RR$. 
\end{proof}



\bcor \label{cor: monotone_driver}
    For any $a \in \RR$, the function $x \mapsto x \Phi(a/x)$ is decreasing for all $x > 0$.     
\ecor
\begin{proof}
    For $a=0$, the function is the constant $0$ and is therefore decreasing. On the other hand, for any $a \in \RR\backslash\{0\}$, letting $f(x) := x \Phi(a/x) = a \Psi(a/x)$ gives $f'(x) = - \frac{a^{2}}{x^{2}} \Psi'(a/x) < 0$ by Lemma \ref{lem: psi_cdf}. 
\end{proof}


The above results guarantee that $\Phi$ is non-positive on $(-\infty, 0)$ and non-negative on $[0,\infty)$ which implies $\Phi(x) \leq x$ for $x\geq0$ since $\Psi$ is a cdf on $\mathbb{R}$. Hence, for any constant $c>0$ and any $x\in \mathbb{R}$, it holds that
\begin{equation}\label{basic:ineq:Phi}
\begin{aligned}
c\Phi(x/c) \leq x^+,
\end{aligned}
\end{equation}
and we have the following lemma.


\bl\label{lemma1.1}
For any $\epsilon \in (0,1)$ and any $c >0$, it holds
\begin{gather}
0\leq x^+ - c\Phi(x/c) \leq \epsilon - c\ln(1-e^{-\epsilon/c})+ c[\ln(|x|)]^+ - c\ln(c). \label{eq1.1}
\end{gather}
\el

\begin{proof}
For simplicity, we set $x'= x/c$ and observe that \eqref{basic:ineq:Phi} gives
\begin{gather*}
0\leq c(x' - \Phi(x')) = x^+ - c\Phi(x/c).
\end{gather*}
On the one hand, for $x> \epsilon$, we have
\begin{align*}
x' - \Phi(x')   & = x'- \ln(e^{x'}- 1) + \ln(x') \\
				& = -\ln\left( 1-e^{-x'}\right) + \ln(x) - \ln(c)\\
			    & \leq -\ln\left( 1-e^{-\frac{\epsilon}{c}}\right) + [\ln(x)]^+ - \ln(c).
\end{align*}
On the other hand, for $0\leq x \leq  \epsilon$, the mean-value theorem together with Lemma \ref{lipphi} and the fact that $\Phi(0)=0$ gives $c(x' - \Phi(x')) = x(1-\Phi'(\tilde{x}))\leq x \leq \epsilon$ for some $\tilde{x} \in  [0, x/c]$. Now, for $x < 0$, we note that
\begin{gather*}
\Phi(x) = \ln \left( \frac{e^{x}-1}{x} \right)  = \ln \left( \frac{1-e^{-|x|}}{|x|} \right)  
\end{gather*}
and, by a similar proof, we obtain
\begin{align*}
0< c\left[ - \Phi(x/c) \right] 
							   & \leq  \epsilon - c\ln(1-e^{-\epsilon/c})+ c[\ln(|x|)]^+ - c\ln(c)
\end{align*}
which concludes the proof.
\end{proof}

\end{document}